\documentclass[reqno,11pt]{amsart}
\usepackage{geometry}

\usepackage{amsmath}
\usepackage{amssymb}
\usepackage{hyperref}
\usepackage{xcolor,tikz}
\usetikzlibrary{calc}
\numberwithin{equation}{section} 
\usepackage[sort]{natbib}
\usepackage{enumerate}

\usepackage{dsfont}

\newtheorem{theorem}{Theorem}[section]
\newtheorem{lemma}[theorem]{Lemma}

\newtheorem{proposition}[theorem]{Proposition}

\theoremstyle{definition}
\newtheorem{definition}[theorem]{Definition}

\newcommand\R{\mathbb{R}}

\newcommand\C{\mathbb{C}}

\newcommand{\qtq}[1]{\quad\text{#1}\quad}

\renewcommand\epsilon{\varepsilon}
  
\usepackage{bbm}
\newcommand{\1}{\mathds{1}}
\newcommand{\Schwartz}{\mathcal{S}}

\let\Re=\undefined\DeclareMathOperator{\Re}{Re}
\let\Im=\undefined\DeclareMathOperator{\Im}{Im}

\newcommand{\rn}[1]{%
  \textup{\uppercase\expandafter{\romannumeral#1}}%
}

\newcommand{\jb}[1]{\langle#1\rangle}

\allowdisplaybreaks

\begin{document}

\title[On ill-posedness for GT]{On ill-posedness for the Gabitov--Turitsyn equation}

\author[M.~Kowalski]{Matthew Kowalski}
\address{Department of Mathematics, University of California, Los Angeles, CA 90095, USA}
\email{mattkowalski@math.ucla.edu}

\begin{abstract}
    We investigate the well- and ill-posedness theory for the Gabitov--Turitsyn equation, which models the long-time dynamics of pulses in dispersion-managed optical fibers. We identify two critical regularities, corresponding to two scaling pseudo-symmetries, that demarcate regimes of ill-posedness. First, we identify $s_m = \tfrac{d}{2} - \tfrac{2}{p},$ coinciding with the monomial NLS. For $s \geq \max(s_m,0)$, local well-posedness is known to hold in $H^s$, while for $s < s_m$, we show that the data-to-solution map fails to be $C^{p+1}$ in $H^s$. Second, we identify $s_i = \tfrac{d}{2} - \tfrac{4}{p},$ below which we conjecture that norm inflation occurs. We resolve this conjecture in $H^s$ in the case $s < \min (s_i, 0)$---specifically for the one-dimensional cubic model---and in the case $1 \leq s < s_i$. In the case $s_i \geq 1$, we establish norm inflation by showing that suitable solutions undergo {\em energy equipartition}: a rapid renormalization of kinetic and potential energy.
\end{abstract}

\maketitle

\section{Introduction}\label{intro}
    \noindent
We investigate the well- and ill-posedness theory of the {\em Gabitov--Turitsyn equation}:
\begin{equation}\label{GT}\tag{GT}\begin{split}
        iu_t + \jb{\gamma} \Delta u + \int_0^1 e^{-i\sigma \Delta} \Big[\big|e^{i\sigma \Delta} u\big|^p \cdot e^{i\sigma \Delta} u \Big]d\sigma = 0, \quad
        u(0,x) = u_0(x),
\end{split}\end{equation}
where $u : (-T, T) \times \R^d_x \to \C$ is a complex function of spacetime and $\jb{\gamma} \neq 0$ is a constant. To justify a power series expansion, we make the assumption that $p \geq 2$ is even.

The equation \eqref{GT} is the primary, effective model of large-scale pulse dynamics in dispersion-managed optical fibers; see Section~\ref{fibers} for further discussion. The constant term $\jb{\gamma}$ is the net dispersion of the fiber, with $\jb{\gamma} > 0$ representing the focusing equation and $\jb{\gamma} < 0$ representing the defocusing. For our purposes, the exact value of $\jb{\gamma}$ is inconsequential; however, it is expected to affect the broader theory. Although not considered here, significant attention has also been paid to the vanishing dispersion, $\jb{\gamma} = 0$ case, which remarkably supports soliton solutions\cite{solitons-zero-dispersion, solitons-zero-dispersion-reason}.

The model \eqref{GT} is a highly nonlocal\footnote{Indeed, due to the linear propagators, \eqref{GT} can be seen as {\em nonlocal-in-time}.} variant of the nonlinear Schr\"odinger equation \eqref{NLS}. As such, it is a Hamiltonian partial differential equation, with mass and energy given by
\begin{equation}\label{energy}
    M(u) = \int_{\R^d} |u|^2 \qtq{and} E(u) = - \tfrac{\jb{\gamma}}{2} \int |\nabla u|^2 + \tfrac{1}{p+2} \iint_0^1 |e^{i \sigma \Delta} u|^{p+2} d\sigma dx.
\end{equation}
We note that this closely resembles the energy of the monomial \eqref{NLS}, only with the potential energy averaged in time.

Central to our narrative is the lack of a scaling symmetry for \eqref{GT}. Indeed, the interval $[0,1]$ in the nonlinearity fixes a fundamental length of time, destroying any genuine scaling symmetries and obscuring the theory of well- and ill-posedness. To surmount this, we identify two {\em scaling pseudo-symmetries} for \eqref{GT},\eqref{intro/monomial} and \eqref{intro/integrated}. These pseudo-symmetries do not preserve the class of solutions, but instead take solutions of \eqref{GT} to solutions of a \eqref{GT}-variant with the integration altered. 

Before detailing the scaling pseudo-symmetries, let us outline our main results. These scaling pseudo-symmetries identify two `critical' regularities: $s_m = \frac{d}{2} - \frac{2}{p}$ and $s_i = \frac{d}{2} - \frac{4}{p}$. Together with the Galilean threshold, $s = 0$, we find that these critical regularities demarcate regimes of well- and ill-posedness for \eqref{GT}:

\subsection*{Outline of main results}
    {\em \begin{enumerate}[(I)]
        \item {Theorem~\ref{intro/lwp}}: {\em For $s \geq \max(s_m, 0)$, \eqref{GT} is locally well-posedness in $H^s$ with an analytic dependence on the initial data and globally well-posed for small initial data.}

        \item {Theorem~\ref{intro/analytic-IP}}:
        {\em For $s < s_m$, we show that the data-to-solution map of \eqref{GT} fails to be smooth from $H^s \to C_t H_x^s$.}

        \item {Theorems \ref{intro/inflation-negative} and \ref{intro/inflation-energy}}:
        {\em For $s < s_i$, we conjecture that norm inflation occurs. This is shown in $H^s$ in the mass-subcritical $s < \min(s_i,0)$ case---specifically for the one-dimensional cubic \eqref{GT}---and energy-supercritical case $1 \leq s < s_i$.}
    \end{enumerate}}

\subsection{Monomial scaling pseudo-symmetry}
    Under the usual scaling symmetry for monomial \eqref{NLS},
    \begin{equation}\label{intro/monomial}
        u(t,x) \;\longmapsto\; \lambda^{-\frac{2}{p}} u\Big(\frac{t}{\lambda^2}, \frac{x}{\lambda}\Big),
    \end{equation}
    solutions of \eqref{GT} will be mapped to solutions of an altered version of \eqref{GT}:
    \begin{equation*}
        iv_t + \jb{\gamma} \Delta v + \lambda^{-2} \int_0^{\lambda^{2}} e^{-i\sigma \Delta} \Big[\big|e^{i\sigma \Delta} v\big|^p \cdot e^{i\sigma \Delta} v \Big]d\sigma = 0.
    \end{equation*}
    Notably, the nonlinearity is still an averaging, but is now over a new interval. Associated to this scaling is a critical regularity\footnote{The subscript $m$ is used to indicate `monomial' for quantities associated to the monomial \eqref{NLS}.} $s_m = \frac{d}{2} - \frac{2}{p}$.
    
    For $s \geq \max(s_m,0)$, it is known that \eqref{GT} is locally well-posed in $H^s$ and globally well-posed for small initial data; see earlier work by Kawakami and Murphy \cite{GT-Jason-small-large-scattering}. To align with our first ill-posedness result, and for the sake of completeness, we present a proof of the following theorem via power-series methods; see, e.g., \cite{series-Christ, series-Tao}.
    \begin{theorem}\label{intro/lwp}
        For $u_0 \in \dot{H}^{s_m \vee 0}$ and $T \sim \|u_0\|^{-p}_{\dot{H}^{s_m \vee 0}}$, there exists a unique solution to \eqref{GT}:
        \begin{equation*}
            u \in C_t \dot{H}^{s_m\vee 0}((-T,T) \times \R^d) \qtq{with} u(0,x) = u_0(x).
        \end{equation*} 
        Moreover, the data-to-solution map is real analytic on a neighborhood of $u_0 = 0$: for $\|u_0\|_{\dot{H}^{s_m \vee 0}} \leq R$ and $T \sim R^{-p}$ the data-to-solution map 
        \begin{equation*}
            u_0 \in B_R\big(\dot{H}^{s_m \vee 0}\big) \quad \longmapsto \quad u \in C_t \dot{H}_x^{s_m \vee 0} \big((-T,T) \times \R^d\big)
        \end{equation*}
        satisfies a power series expansion; see \eqref{lwp/power-series}.
    
        In addition, there exists $\delta = \delta(p,d)$ such that for all $\|u_0\|_{\dot{H}^{s_m \vee 0}} < \delta$, the associated solution $u$ may be extended to a global solution in $C_t \dot{H}_x^{s_m\vee 0}(\R \times \R^d)$ which scatters.
    \end{theorem}
    Curiously, this theorem is stronger than the local well-posedness theory for the monomial \eqref{NLS}. First, the dependence of $T$ solely on $\dot{H}^{s_m \vee 0}$ is reminiscent of a sub-critical model. For critical models, in contrast, we expect the time of existence to depend on the profile of the initial data. In addition, we avoid the use of auxiliary spaces and work directly in $L_t^\infty H_x^{s_m \vee 0}$, giving unconditional uniqueness. Analytically, these benefits arise from the averaging, which allows for Strichartz estimates to be used directly on the nonlinearity; see the progression from \eqref{lwp/1} to \eqref{lwp/2}. 
    
    Prior work by Choi, Hundertmark, and Lee \cite{GT-Choi-wp} considered high-regularity well-posedness of \eqref{GT} with a focus on remarkably general nonlinearities, in the sense of \eqref{fibers/GT}.
    
    For $s < s_m$, we find that the data-to-solution map of \eqref{GT} fails to be $C^{p+1}$ from $H^s \to C_t H_x^s$ in a Fr\'echet sense.
    As contraction mapping arguments for dispersive PDEs naturally achieve an analytic dependence on the initial data, this indicates that modern (non-integrable) well-posedness tools fail for \eqref{GT} for $s < s_m$. In this sense, $s_m$ forms a threshold for the well-posedness theory of \eqref{GT}.
    
    We wish to be careful with the statement of this result. To clarify our meaning, we briefly recall the argument of Bourgain in \cite[\S 6]{series-Bourgain}: For $\phi \in \Schwartz(\R^d)$, consider the initial value problem \eqref{GT} with initial data $u_0 = \delta \phi$ for some $\delta > 0$. Let $u = u(\delta,\phi)$ be the corresponding local solution (given, for instance, by Theorem~\ref{intro/lwp}). Then, by iterating the Duhamel formula, we find that
    \begin{equation*}
        \frac{\partial^{p+1} u}{\partial \delta^{p+1}} \bigg|_{\delta = 0} = i \int_0^t \int_0^1 e^{i\jb{\gamma}(t-s)\Delta -i\sigma\Delta} \Big[\big|e^{i(s + \sigma)\Delta} \phi\big|^p \cdot e^{i(s+\sigma)\Delta} \phi \Big] d\sigma ds = \Xi_1(\phi).
    \end{equation*}
    Which is to say that $\Xi_1(\phi)$ is the $(p+1)^{\text{st}}$ directional derivative of the data-to-solution map at initial data $0$ in the direction of $\phi$. We then find that:
    \begin{theorem}\label{intro/analytic-IP}
        For any $T > 0$ and $s < s_m$, the $(p+1)^{\text{st}}$ derivative of the data-to-solution map, $\Xi_1$, fails to be bounded $H^s \to L_t^\infty H_x^s((-T,T))$ for Schwartz data.
    \end{theorem}

\subsection{Integrated scaling pseudo-symmetry}
    If we alter \eqref{GT} so that the nonlinearity is integrated over $[0,\infty)$ (or $\R$), we recover a true scaling symmetry:
    \begin{equation}\label{intro/integrated}
        u(t,x) \;\longmapsto\; \lambda^{-\frac{4}{p}} u\Big(\frac{t}{\lambda^2}, \frac{x}{\lambda}\Big).
    \end{equation}
    We note that this altered model is natural to consider mathematically, but does not appear to arise in the study of dispersion-managed optical fibers; see Section~\ref{fibers}. When applied to solutions of \eqref{GT}, the rescaling \eqref{intro/integrated} will map solutions of \eqref{GT} to solutions of a \eqref{GT}-variant,
    \begin{equation*}
        iv_t + \jb{\gamma} \Delta v + \int_0^{\lambda^{2}} e^{-i\sigma \Delta} \Big[\big|e^{i\sigma \Delta} v\big|^p \cdot e^{i\sigma \Delta} v \Big]d\sigma = 0.
    \end{equation*}
    Notably, the region of integration is again altered, but the nonlinearity is no longer averaged.
    
    Associated to this pseudo-symmetry is a critical regularity\footnote{The subscript $i$ is used to indicate `integrated' for quantities associated to the integrated nonlinearity of \eqref{GT}.} $s_i = \frac{d}{2} - \frac{4}{p}$.
    Below this threshold, we conjecture that \eqref{GT} is ill-posed and exhibits norm inflation in $H^s$ for $s < s_i$. Specifically, we expect that for all $s < s_i$, the following occurs:
    \begin{definition}[Norm inflation in $H^s$]
        For all $\epsilon > 0$, there exists smooth initial data $u_0$ with corresponding solution $u(t)$ to \eqref{GT} that satisfies
        \begin{equation*}
            \|u_0\|_{\dot{H}^s} < \epsilon \qtq{with} \|u(T)\|_{\dot{H}^s} > \epsilon^{-1},\qtq{for some} |T| < \epsilon.
        \end{equation*}
    \end{definition}

    We prove norm inflation in $H^s$ and resolve this conjecture in two cases: the mass-subcritical regime $s < \min(0,s_i)$ and the energy-supercritical regime $1 \leq s < s_i$.
    
    In the mass-subcritical regime, we build on the norm inflation results for \eqref{NLS} based on the power-series representation of the data-to-solution map \cite{series-Oh, series-Christ,series-Kishimoto, series-Tao, series-Bourgain}. In particular, we expand the methods of Kishimoto \cite{series-Kishimoto} and Oh \cite{series-Oh}. For simplicity of exposition,  we consider only the one dimension cubic \eqref{GT} and show norm inflation in $H^s$ for $s < s_i = -\frac{3}{2}$:
    \begin{theorem}\label{intro/inflation-negative}
        For the one dimensional cubic \eqref{GT}, norm inflation occurs in $H^s$ for $s < s_i = -\frac{3}{2}$. 
        As a consequence, for any $T > 0$, the data-to-solution map of \eqref{GT} fails to be continuous at $u_0 = 0$ from $H^s \to C_t H_x^s((-T,T))$ for $s < s_i = -\frac{3}{2}$.
    \end{theorem}
    \noindent We expect that this argument extends naturally---with the methods in \cite{series-Oh}---to all \eqref{GT} models for $s < \min(s_i, 0)$. We do not pursue the details here.

    In contrast, for positive regularities, we must derive novel methods to show norm inflation. The standard method of proving norm inflation for \eqref{NLS}, due to Christ, Colliander, and Tao in \cite{CCT-ip-1,CCT-ip-2}, relies on the vanishing dispersion limit to yield approximate solutions to \eqref{NLS} with well-understood behaviors. For \eqref{GT}, this method appears intractable: In the vanishing dispersion limit of \eqref{GT}, approximate solutions are not known to exist. This is due to the nonlocality of the nonlinearity. One may remove the nonlocality of \eqref{GT} in the large scale limit, such as in the work \cite{GT-Jason-large-scale}, however this is at odds with the subsequent small scales needed for norm inflation.

    To surmount this, we instead build on structural identities of \eqref{GT}, namely the virial identity. This is originally due to Glassey \cite{NLS-virial-glassey} for \eqref{NLS} and first shown for \eqref{GT} by Choi, Hong, and Lee \cite{GT-dichotomy}. For the sake of completeness, we present and prove the virial identity for general $\jb{\gamma}$, $p$, and $d$. For simplicity, we restrict to Schwartz initial data; see \cite{GT-dichotomy} for a more precise statement of the necessary regularity and decay.
    \begin{proposition}[Virial identity]\label{intro/virial}
        Let $u$ be the maximal solution of \eqref{GT} for $p \geq \frac{8}{d}$, with initial data $u_0 \in \Schwartz(\R^d)$. Define the variance of $u$ as
        \begin{equation*}
            v(t) = \int_{\R^d} |x|^2 |u(t,x)|^2 dx.
        \end{equation*}
        
        In the defocusing case, $\jb{\gamma} = -1$, we find that for $-\frac{1}{2} \leq t \leq 0$,
        \begin{equation*}
            v(t) \leq v(0) - \dot{v}_1(0)t + C(p,d) E(u) t^2,
        \end{equation*}
        where $C(p,d) > 0$ for $p \geq \frac{8}{d}$. 
        
        In the focusing case, $\jb{\gamma} = 1$, we find that for $t < 0$,
        \begin{equation*}
            v(t) \leq v(0) + \dot{v}_1(0) t - 8E(u) t^2.
        \end{equation*}
        
        In both cases, $\dot{v}_1$ is given by and, for $t < 0$, satisfies
        \begin{equation*}
            \dot{v}_1(t) = \int\overline{u} (x \cdot \nabla u) dx \geq \dot{v}_1(0) - 16 E(u) t.
        \end{equation*}
    \end{proposition}
    
    In the defocusing case, $\jb{\gamma} = -1$, we use this identity to prove that suitable solutions undergo {\em energy equipartition}, a rapid renormalization of kinetic energy and potential energy:
    \begin{proposition}[Energy equipartition]\label{intro/energy-equipartition}
        Suppose that $u_0$ is real-valued and sufficiently regular to justify the virial identity \eqref{virial}. Then for $T^2 \sim E(u_0)/\int |x u_0|^2 dx \ll 1$, $T < 0$, the corresponding solution $u$ to \eqref{GT} for $\jb{\gamma} = -1$ satisfies
        \begin{equation*}
            \|u(T)\|_{\dot{H}_x^1}^2 \gtrsim \|e^{i\sigma \Delta} u(T)\|^{p+2}_{L^{p+2}_{\sigma, x}([0,1])}.
        \end{equation*}
        In other words, should $\|u_0\|^2_{H^1} \ll \|e^{i\sigma \Delta} u_0\|^{p+2}_{L^{p+2}_{\sigma, x}([0,1])}$, then the kinetic and potential energy of $u$ become comparable by time $T$.
    \end{proposition}
    
    By constructing initial data $u_0$ which satisfies 
    \begin{equation*}
        \|u_0\|_{H^1} \ll 1, \quad \|e^{i\sigma \Delta} u_0\|^{p+2}_{L^{p+2}_{\sigma, x}([0,1])} \gg 1, \qtq{and}\frac{E(u_0)}{\int |x u_0|^2 dx} \ll 1,
    \end{equation*}
    the preceding proposition immediately implies a rapid growth of the $H^1$ norm and hence norm inflation in $H^s$ for $s \geq 1$.

    In the focusing case, $\jb{\gamma} = 1$, we instead follow the work of Choi, Hong, and Lee \cite{GT-dichotomy}, which showed finite time blowup. Together, these methods yield the follow theorem:
    \begin{theorem}\label{intro/inflation-energy}
        Suppose that $s_i > 1$. Then for $1 \leq s < s_i$, norm inflation occurs in $H^s$ for \eqref{GT}.
        As a consequence, for any $T > 0$, the data-to-solution map of \eqref{GT} fails to be continuous at $u_0 = 0$ from $H^s \to C_t H_x^s((-T,T))$ for $1 \leq s < s_i$.
    \end{theorem}

    \subsection{Connection to fiber-optics}\label{fibers}
        Accounting for the dominant effects, the standard model for the evolution of a pulse in an optical fiber is the cubic nonlinear Schr\"odinger equation:
\begin{equation}\label{NLS}\tag{NLS}
    i\partial_t \psi + \gamma \Delta\psi + |\psi|^2 \psi = 0, \quad \psi(0,x) = \psi_0(x), \quad \psi : \R_t \times \R_x \to \C,
\end{equation}
where $\gamma$ is the group velocity dispersion of the fiber (here just called dispersion), and $\psi$ is the complex modulation of a quasi-monochromatic carrier wave. The cubic, or {\em Kerr}, nonlinearity is typically dominant and arises from pulses polarizing the material they propagate within.
We note that the role of $t$ and $x$ are flipped from expectation: $t$ represents the {\em distance} along the optical fiber while $x$ represents a {\em retarded time}, traveling with the carrier wave.

In a typical optical fiber, dispersion dominates the nonlinear effects and causes pulses to broaden. This limits bandwidth as pulses overlap and interact. A common technique to mitigate these effects, introduced in \cite{origin-linear-dispersion-management}, is dispersion-management: concatenating fiber segments with opposite dispersion so that higher-frequency components propagate faster in one section and slower in the next. 
Mathematically, this corresponds to allowing $\gamma = \gamma(t)$ in \eqref{NLS} to alternate between positive and negative along the fiber (i.e. in $t$). This results in the so-called {\em dispersion-managed nonlinear Schr\"odinger equation}.

Strong dispersion-management takes this to the extreme, alternating quickly between extreme positive dispersion and extreme negative. This achieves a low (but nonzero) average dispersion with high local dispersion. Mathematically, this corresponds to $$\gamma(t) = \langle \gamma \rangle + \epsilon^{-1} \gamma_0(t/\epsilon),$$ where $\langle\gamma\rangle$ is the net dispersion and $\gamma_0$ is a periodic function with mean $0$. Crucially, this results in a fiber where dispersion and nonlinear effects are balanced, allowing for stable, soliton-like pulses. 

In the limit as $\epsilon \to 0$, with a suitable change of variables, the one-dimensional cubic Gabitov--Turitsyn equation ($d = 1, p = 2$) emerges as the natural model of the large-scale dynamics:
\begin{equation}\label{fibers/GT}
        iu_t + \jb{\gamma} \Delta u + \int_\R e^{-i\sigma \Delta} \Big[\big|e^{i\sigma \Delta} u\big|^2 \cdot e^{i\sigma \Delta} u \Big]d\mu(\sigma) = 0,
\end{equation}
where $\mu$ is a probability measure, dictated by fiber properties and the fine details of the dispersion management.
See \cite{GT-justification-1,GT-justification-2,GT-justification-3, GT-justification-4, solitons-zero-dispersion-reason} for the numerical and rigorous justification of this limiting process and see \cite{T-averaging} for the general Hamiltonian process. We also direct the reader to \cite{solitons-positive-dispersion} for a nice mathematical exposition to \eqref{fibers/GT} and discussion of the measure $\mu$.

To recover the specific form of \eqref{GT} used in this paper, a model case with $\gamma_0 = \1_{[0,1)} - \1_{[1,2)}$ (extended periodically) is considered. Other choices of $\gamma_0$ give different probability measures $\mu$ in the nonlinearity of \eqref{fibers/GT}. Indeed, one can even consider fiber loss and the correcting amplification by generalizing $\mu$; see \cite{GT-with-amplification}.

It is important to note that the techniques used in fiber optics have experienced dramatic twists in recent history. Dispersion-management alone underwent a number of iterations---such as focusing on vanishing net dispersion---before strong dispersion management was adopted; see \cite{DCF-history,DSP-DCF-comparison}. We direct the interested reader to \cite[\S2.1]{dispersion-management-survey} for a thorough history and explanation of various methods. 

In recent years, {\em digital signal processing} (DSP), enabled by coherent detection, has emerged as a competitor to dispersion-management. DSP sidesteps physical compensation by propagating pulses unmanaged and digitally reversing ill-effects at the receiver. This hinges on the long-time dynamics of \eqref{NLS}.
However, as the scale and rate of information transfer increases, DSP has its own challenges (see \cite{DSP-concerns-Turitsyn}). Recent research and industry perspectives have proposed hybrid dispersion-management and DSP approaches; see \cite{DSP-DM-combo-2023,blog-post}. Indeed, many transoceanic systems will continue to rely on dispersion-management due to the cost and time associated with changing to a pure DSP system.

We direct the interested reader to \cite{Biswas, Fibich} for nice textbook derivations of \eqref{NLS} and \eqref{fibers/GT}.
    
    \subsection*{Acknowledgements}
    The author was supported, in part, by NSF grants DMS-2154022, DMS-2452346, and DMS-2348018. 
    The author is grateful to Rowan Killip and Monica Vi\c{s}an for their confidence and guidance. The author is also grateful to Jason Murphy for the introduction to, and discussions on, these interesting models.

    \subsection*{Notation}
        We use the notation $A \lesssim B$ to indicate that $A \leq C B$ for some universal constant $C > 0$ that will change from line to line. If the implied constant is very small, then we use the notation $A \ll B$. If both $A \lesssim B$ and $B \lesssim A$ then we use the notation $A \sim B$. 
When the implied constant fails to be universal, the relevant dependencies will be indicated within the text or included as subscripts on the symbol.

For compactness of notation, we denote the maximum and minimum of two numbers $a$ and $b$ as $a \vee b$ and $a \wedge b$ respectively. For the sake of the reader's intuition, this aligns with the union and intersection of sets; that is to say
\begin{align*}
    (-\infty, a] \cup (-\infty, b] = (-\infty, \max(a,b)] = (-\infty, a \vee b], \\
    (-\infty, a] \cap (-\infty, b] = (-\infty, \min(a,b)] = (-\infty, a \wedge b].
    \end{align*}

Our conventions for the Fourier transform are
\begin{equation*}
\widehat{f}(\xi) = \tfrac{1}{\sqrt{2\pi}}\int e^{-i \xi x} f(x) dx \quad \text{so} \quad f(x) = \tfrac{1}{\sqrt{2\pi}} \int e^{i \xi x} \widehat{f}(\xi) d\xi.
\end{equation*}
This Fourier transform is unitary on $L^2$ and yields the standard Plancherel identities.
When a function $f(t,x)$ depends on both time and space, we let $\widehat{f}(t,\xi)$ denote the Fourier transform of $f$ in only the spatial variable.

For $s > -\frac{d}{2}$, we define the homogeneous Sobolev space $\dot{H}^s$ as the completion of the Schwartz functions $\Schwartz(\R^d)$ with respect to the norm
\begin{equation*}
    \|f\|^2_{\dot{H}^s} = \int_{\R^d} |\xi|^{2s} |\widehat{f}(\xi)|^2 d\xi.
\end{equation*}
For $s \in \R$, we similarly define the inhomogeneous Sobolev space $H^s$ as the completion of the Schwartz functions $\Schwartz(\R^d)$ with respect to the norm
\begin{equation*}
    \|f\|^2_{H^s} = \int_{\R^d} (1 + |\xi|^2)^s |\widehat{f}(\xi)|^d d\xi.
\end{equation*}

We use $L_t^p L_x^q(T \times X)$ to denote the mixed Lebesgue spacetime norm
\begin{equation*}
    \|f\|_{L_t^p L_x^q(T\times X)} = \big\| \|f(t,x)\|_{L^q(X,dx)} \big\|_{L^p(T,dt)} = \bigg[ \int_T \bigg(\int_X |f(t,x)|^q dx\bigg)^{p/q} dt\bigg]^{1/p}.
\end{equation*}
When $p = q$, we let $L^p_{t,x} = L_t^p L_x^p$. When $X = \R^d$, we let $L_t^pL_x^q(T) = L_t^pL_x^q(T\times\R^d)$. 
As an extension of this notation, we use $L_\sigma^r L_t^p L_x^q(I \times T \times X)$ to denote the mixed norm
\begin{equation*}
    \|f(\sigma, t, x)\|_{L_\sigma^r L_t^p L_x^q(I \times T \times X)} = \Big\| \big\| \| f(\sigma, t,x)\|_{L_x^q(X,dx)} \big\|_{L_t^p(T,dt)}\Big\|_{L_\sigma^q(I)}.
\end{equation*}
When $I = [0,1]$ and $X = \R^d$, we let $L_\sigma^r L_t^p L_x^q(T) = L_\sigma^r L_t^p L_x^q(I \times T \times \R^d)$. We extend this notation to Sobolev spaces in the obvious way.

For our purposes, the spatial norm will always be taken over $X = \R^d$.

\section{Preliminaries and linear theory}
    It will be useful throughout the paper to recall a notion of Schr\"odinger-admissible pairs:
\begin{definition}[Schr\"odinger-admissible]
    Fix a spatial dimension $d \geq 1$. We say that a pair $(p,q)$ is \emph{Schr\"odinger-admissible} if
    \begin{equation*}\label{SAP}
        2 \leq p, q \leq \infty, \quad \tfrac{2}{p} + \tfrac{d}{q} = \tfrac{d}{2}, \qtq{and} (p,q,d) \neq (2,\infty,2).
    \end{equation*}
\end{definition}

To adapt to the unique structure of \eqref{GT}, we recall the so-called shifted Strichartz estimates, original appearing in \cite{GT-Jason-small-large-scattering} and inspired by the work in \cite{Kawakami}:
\begin{proposition}[Shifted Strichartz estimates]\label{prel/shifted-Strichartz}
    Let $(p_i,q_i)$ for $i \in \{1,2\}$ be Schr\"odinger admissible pairs with $p_i \neq 2$. Then for any time interval $I \ni 0$, we find
    \begin{equation*}
        \bigg\|\int_0^t e^{i(t-s+\theta - \sigma) \Delta} F(\sigma,s) ds\bigg\|_{L_t^{p_1} L_x^{q_1}(I)} \lesssim_{p,q} \| F(\sigma, t)\|_{L_t^{p_2'} L_x^{q_2'}(I)}.
    \end{equation*}
\end{proposition}
We note that in \cite{GT-Jason-small-large-scattering}, these estimates are only presented in the symmetric $p_1 = p_2$ case and are further presented with Lorentz spaces included. The (non-Lorentz) generalization to $p_1 \neq p_2$ follows the standard argument via duality with only minor changes. For the sake of brevity, we do not present a proof here; instead, we direct the reader to \cite{NLS-clay-lecture} and references therein for a concise presentation of the usual argument.

\section{Analytic local well-posedness}
    In this section, we prove Theorem~\ref{intro/lwp} and show that \eqref{GT} is locally well-posed with a data-to-solution map that is real analytic. To make this analytic dependence explicit, and to align with our ill-posedness results, we present the proof via power-series methods. These methods align closely with \cite[\S 3]{series-Tao}; see also \cite{series-Oh,series-Christ} for an elegant presentation via $(p+1)$-ary trees. As this is a standard argument, we focus on the quantitative well-posedness estimates and keep the details brief.

\subsection{Definition of the power series}
    We begin by briefly defining the power-series expansion of the data-to-solution map. Recall the Duhamel formula for a solution of \eqref{GT}:
    \begin{equation*}
        u(t) = e^{it\jb{\gamma}\Delta} u_0 + i \int_0^t \int_0^1 e^{i\jb{\gamma}t\Delta-i(s+\sigma)\Delta}\Big[\big|e^{i\sigma \Delta} u(s)\big|^p \cdot e^{i\sigma \Delta} u(s) \Big] d\sigma ds.
    \end{equation*}
    Fix some $R > 0$ and some $T = T(R)$ to be chosen later. Let $D = B_R(\dot{H}^{s_m \vee 0})$ denote the space of initial data and $S = C_t \dot{H}_x^{s_m \vee 0}((-T,T) \times \R^d)$ denote the space of solutions, aligning with the desired spaces for Theorem~\ref{intro/lwp}.
    
    We decompose the Duhamel formula into the linear term $L : D \to S$ and the nonlinear correction $N_p : S^{p+1} \to S$, which we will show are well-defined in Proposition~\ref{lwp/quantitative}. Specifically,
    \begin{align*}
        Lf & = e^{it \jb{\gamma}\Delta} f, \\
        N_p(f_0,\cdots, f_p) & = i \int_0^t \int_0^1 e^{i\jb{\gamma}(t-s)\Delta - i\sigma\Delta}\Big[e^{i\sigma \Delta} f_0(s)\cdot\overline{e^{i\sigma \Delta} f_1(s)} \cdot \ldots \cdot e^{i\sigma \Delta} f_p(s)\Big] d\sigma ds.
    \end{align*}
    Note that $N_p$ is a $(p+1)$-linear operator.\footnote{We take this to mean that it is linear {\em or conjugate linear} in each argument.} With this notation, a solution $u$ of \eqref{GT} with initial data $u_0$ satisfies
    \begin{equation*}
        u = L u_0 + N_p(u, \cdots, u).
    \end{equation*}
    Ouroborically substituting this formula into itself, we find the formal expansion
    \begin{equation}\label{lwp/expansion}
        u = L u_0 + N_p(Lu_0, \cdots, Lu_0) + N_p\big[Lu_0,\cdots,Lu_0,N_p(Lu_0,\cdots,Lu_0)\big] + \cdots.
    \end{equation}
    
    Grouping terms of equal order, we then recursively define
    \begin{equation}\label{lwp/definition}\begin{split}
        \Xi_0(u_0) & = L(u_0), \\
        \Xi_j(u_0) & = \sum_{\substack{j_0,\dots,j_{p} \geq 0 \\ j_0 + \dots + j_{p} = j - 1}} N_p(\Xi_{j_0}(u_0), \dots, \Xi_{j_{p}}(u_0)). 
    \end{split}\end{equation}
    We note here that $\Xi_j$ is a $(jp + 1)$-linear operator; the index $j$ is used to indicate the total `depth' of the recursion, how many copies of $N_p$ appear within any individual summand.
    This implies that $u$ has the following (formal) power series expansion
    \begin{equation}\label{lwp/power-series}
        u = \sum_{j \geq 0} \Xi_j(u_0).
    \end{equation}
    
    Should this power series converge absolutely in $S$, standard arguments then imply that $u$ solves \eqref{GT} in a strong sense and that the data-to-solution map is analytic $D \to S$.
\subsection{Quantitative local well-posedness}
To formalize this, we show that \eqref{lwp/power-series} converges absolutely in $S$. To do so, we prove a quantitative well-posedness in the sense of \cite{series-Tao}. From \cite{series-Tao}, these estimates are sufficient to conclude local well-posedness; we present an overview of the logic following the proposition.
\begin{proposition}[Quantitative well-posedness]\label{lwp/quantitative}
    The operators $N_p: S^{p+1} \to S$ and $L : D \to S$ are bounded in the following sense:
    \begin{align*}
        \|Lg\|_S & = \|g\|_D \\
        \|N_p(f_0,\dots,f_p)\|_S & \leq C_p T \|f_0\|_S \dots \|f_p\|_S.
    \end{align*}
\end{proposition}
\begin{proof}
    The linear estimate follows immediately from the fact that the linear Schr\"odinger equation is an isometry on all $L^2$-based Sobolev spaces. We thus focus our attention on the nonlinear correction. 
    
    We begin with the case $s_m \geq 0$, i.e. $p \geq \frac{4}{d}$. Using the standard Strichartz estimates and a Sobolev embedding, we argue directly to estimate $N_p$:
    \begin{align}
        \|N_p(f_0,\dots,f_p)&\|_{L_t^\infty \dot{H}_x^{s_m}(-T,T)} \nonumber\\
        & = \bigg\| \int_0^t \int_0^1 e^{i\jb{\gamma}(t-s)\Delta -i\sigma\Delta}\Big[e^{i\sigma \Delta} f_0(s) \cdot \ldots \cdot e^{i\sigma \Delta} f_p(s)\Big] d\sigma ds\bigg\|_{L_t^\infty \dot{H}_x^{s_m}(-T,T)} \label{lwp/1}\\
        %& \leq \bigg\| \int_0^t \Big\|e^{i\sigma \Delta} f_0(s) \cdot \ldots \cdot e^{i\sigma \Delta} f_p(s)\Big\|_{L_\sigma^1\dot{H}_x^{s_m}([0,1])} ds\bigg\|_{L_t^\infty(-T,T)} \\
        & \lesssim T\Big\|e^{i\sigma \Delta} f_0(s) \cdot \ldots \cdot e^{i\sigma \Delta} f_p(s)\Big\|_{L_s^\infty L_\sigma^1\dot{H}_x^{s_m}((-T,T) \times [0,1])} \nonumber\\
        & \leq T\sum_{j=0}^{p}\big\||\nabla|^{s_m}e^{i\sigma \Delta} f_j(s)\big\|_{L_s^\infty L_\sigma^{p+1} L_x^{\frac{2d(p+1)}{d(p+1)-4}}}\prod_{k \neq j}\big\|e^{i\sigma \Delta} f_k(s)\big\|_{L_s^\infty L_\sigma^{p+1} L_x^{\frac{dp(p+1)}{2}}} \nonumber\\
        & \lesssim T\prod_{k}\big\||\nabla|^{s_m}e^{i\sigma \Delta} f_k(s)\big\|_{L_s^\infty L_\sigma^{p+1} L_x^\frac{2d(p+1)}{d(p+1)-4}} \label{lwp/2}\\
        & \lesssim T\prod_{k}\big\|f_k\big\|_{L_s^\infty \dot{H}_x^{s_m}((-T,T))}. \nonumber
    \end{align}
    Where we note that
    \begin{equation*}
        \big(p+1, \tfrac{2d(p+1)}{d(p+1)-4}\big)
    \end{equation*}
    is a Schr\"odinger-admissible pair and that $\frac{2d(p+1)}{d(p+1)-4} < \infty$ for $p \geq \frac{4}{d}$.

    In the case $s_m < 0$ (i.e. $p < \frac{4}{d}$), we may argue similarly to bound the nonlinear correction as
    \begin{align*}
        \|N_p(f_0, \dots, f_p)\|_{L_t^\infty L_x^2(-T,T)} 
        & \lesssim T\Big\|e^{i\sigma \Delta} f_0(s) \cdot \ldots \cdot e^{i\sigma \Delta} f_p(s)\Big\|_{L_s^\infty L_\sigma^1 L_x^2((-T,T) \times [0,1])} \\
        & \lesssim T\prod_{k}\big\|e^{i\sigma \Delta} f_k(s)\big\|_{L_s^\infty L_\sigma^{p+1} L_x^{2(p+1)}((-T,T) \times [0,1])}.
    \end{align*}
    We note that $p < \frac{4}{d}$ and so $p+1 \leq \frac{4(p+1)}{dp}$. Embedding $L_\sigma^\frac{4(p+1)}{dp}([0,1]) \hookrightarrow L_\sigma^{p+1}([0,1])$ and noting that $\big(\frac{4(p+1)}{dp}, 2(p+1)\big)$ is a Schr\"odinger-admissible pair, we may then estimate
    \begin{align*}
        \|N_p(f_0,\dots,f_p)\|_{L_t^\infty L_x^2(-T,T)} 
        & \lesssim T \prod_{k}\big\|e^{i\sigma \Delta} f_k(s)\big\|_{L_s^\infty L_\sigma^{\frac{4(p+1)}{dp}} L_x^{2(p+1)}((-T,T) \times [0,1])} \\
        & \leq T \prod_{k}\big\|f_k\big\|_{L_s^\infty L_x^2((-T,T))}.
    \end{align*}
    This concludes the proof of Proposition~\ref{lwp/quantitative}.
\end{proof}

From the preceding proposition and the definition \eqref{lwp/definition}, we find that
\begin{equation}\label{lwp/3}
    \|\Xi_j(u_0)\|_{S} \leq (C_p T)^j \|u_0\|_{D}^{jp+1},
\end{equation}
for a potentially new constant $C_p$. Indeed, by construction, we already know that $\Xi_j = 0$ for $n \neq 1 \mod{p}$ and so it suffices to consider $n = 1 \mod{p}$. The justification of the constant $C_p^j$ requires a combinatorial trick which we will employ later; see \eqref{negative/upper/claim} and the progression from \eqref{negative/5} to \eqref{negative/6}. See also \cite[Lemma 2.3]{series-Oh} for the same combinatorial trick in the context of bounding the number of $(p+1)$-ary trees.

The estimate \eqref{lwp/3} then implies that the series \eqref{lwp/power-series} converges absolutely in $S$ provided that $T \|u_0\|_D^{p}\ll_p 1$. Standard arguments then imply that this power series is a solution to \eqref{GT} in the strong sense. Alternatively, the quantitative well-posedness estimates, Proposition~\ref{lwp/quantitative}, can be used to complete the usual contraction mapping argument.

\subsection{Global existence for small initial data}
Finally, we turn our attention to showing global well-posedness for small initial data. By the local well-posedness theory established above, it suffices to show that the $\dot{H}^{s_m \vee 0}$ norm of the solution remains bounded. As this appeared previously in \cite{GT-Jason-small-large-scattering}, we keep the details brief.
\begin{proof}
    We restrict attention to $s_m \geq 0$ as the case $s_m \leq 0$ follows immediately from the conservation of mass. 
    Suppose that a solution $u$ exists on the time interval $(-T,T)$. It then suffices to show that $\|u\|_{\dot{H}^{s_m}}$ is bounded on $(-T,T)$, uniformly in $T$. To do so, we first show bounds in general Strichartz spaces, before specializing.
    
    For ease of notation, fix the Schr\"odinger-admissible pair
    \begin{equation*}
        (q_0, r_0) = \Big(p+2, \frac{2d(p+2)}{d(p+2) - 4}\Big),
    \end{equation*}
    and note that $\frac{2d(p+2)}{d(p+2) - 4} < \infty$ because $p \geq \frac{4}{d}$. By the shifted Strichartz estimates, Lemma \ref{prel/shifted-Strichartz}, for all Sch\"odinger-admissible pairs $(q,r)$ with $q \neq 2$, we may estimate
    \begin{align*}
        \big\|& |\nabla|^{s_m} e^{i \tau \Delta} u\big\|_{L_\tau^\infty L_t^q L_x^r([0,1] \times (-T,T))}\\
        & \lesssim \|u_0\|_{\dot{H}^{s_m}} + \bigg\||\nabla|^{s_m}\int_0^t \int_0^1 e^{i\jb{\gamma}(t-s)\Delta + i(\tau - \sigma)\Delta} \Big[|e^{i\sigma \Delta}u(s)|^p \cdot e^{i\sigma \Delta}u(s)\Big] d\sigma ds \bigg\|_{L_\tau^\infty L_t^q L_x^r([0,1] \times (-T,T))}  \\
        & \leq \|u_0\|_{\dot{H}^{s_m}} + \bigg\||\nabla|^{s_m}\int_0^t e^{i\jb{\gamma}(t-s)\Delta + i(\tau - \sigma)\Delta} \Big[|e^{i\sigma \Delta}u(s)|^p \cdot e^{i\sigma \Delta}u(s)\Big] ds \bigg\|_{L_{\sigma,\tau}^\infty L_t^q L_x^r([0,1]^2 \times (-T,T))}  \\
        & \lesssim \|u_0\|_{\dot{H}^{s_m}} + \bigg\||\nabla|^{s_m} \Big[|e^{i\sigma \Delta}u|^p \cdot e^{i\sigma \Delta}u\Big]\bigg\|_{L_{\sigma}^\infty L_t^{q_0'} L_x^{r_0'}([0,1]\times(-T,T))}  \\
        & \lesssim \|u_0\|_{\dot{H}^{s_m}} + \big\||\nabla|^{s_m} e^{i\sigma \Delta}u\big\|_{L_{\sigma}^\infty L_t^{q_0} L_x^{r_0}([0,1]\times(-T,T))} \big\|e^{i\sigma \Delta}u\big\|^p_{L_{\sigma}^\infty L_t^{q_0} L_x^{\frac{2dp(p+2)}{8}}([0,1]\times(-T,T))}  \\
        & \lesssim \|u_0\|_{\dot{H}^{s_m}} + \big\||\nabla|^{s_m} e^{i\sigma \Delta}u(s)\big\|^{p+1}_{L_{\sigma}^\infty L_t^{q_0} L_x^{r_0}([0,1]\times(-T,T))}.
    \end{align*}
    
    First, take $q = q_0$ and $r = r_0$. If $\|u_0\|_{\dot{H}^{s_m}} \ll_{p} 1$, then a standard continuity argument implies 
    \begin{equation*}
        \big\||\nabla|^{s_m} e^{i\tau \Delta} u\big\|_{L_\tau^\infty L_t^{q_0} L_x^{r_0}([0,1]\times(-T,T))} \leq 2 \|u_0\|_{\dot{H}^{s_m}}.
    \end{equation*}
    Repeating the above calculation with $q = \infty$ and $r = 2$ then implies that 
    \begin{equation*}
        \|e^{i\tau \Delta} u\|_{L_{\tau,t}^\infty \dot{H}_x^{s_m}([0,1]\times(-T,T))} \leq C\big(\|u_0\|_{\dot{H}^{s_m}}\big),
    \end{equation*}
    uniformly in $T$. Given the local well-posedness theory, this concludes the proof of global well-posedness for small initial data in $\dot{H}^{s_m}$. As this gives global spacetime bounds, scattering then follows from an extension of the canonical argument; this is developed in \cite{GT-Jason-small-large-scattering}.
\end{proof}

\section{Failure of analytic well-posedness}
    In this section, we prove Theorem~\ref{intro/analytic-IP} and show that the data-to-solution map fails to be analytic below $H^{s_m}$. Here, and in the section to follow, we restrict our attention to the inhomogeneous Sobolev spaces $H^s$ as they represent the stronger ill-posedness results. 

To show that the data-to-solution map fails to be $C^{p+1}$ from $H^s \to C_t H_x^s$ for $s < s_m$, it suffices to show that the first nonlinear term in \eqref{lwp/power-series} is unbounded for Schwartz initial data. We analyze this term directly. For ease of notation, we restrict attention to the focusing, $\jb{\gamma} = 1$ case. Our methods extend naturally to any $\jb{\gamma} \neq 0$.

\begin{proof}[Proof of Theorem~\ref{intro/analytic-IP}]
    We analyze the first nonlinear term $\Xi_1$. Recall that
    \begin{equation*}
        [\Xi_1(\phi)](t,x) = i \int_0^t \int_0^1 e^{i(t-s-\sigma)\Delta} \Big[\big|e^{i(s + \sigma)\Delta} \phi\big|^p \cdot e^{i(s+\sigma)\Delta} \phi \Big] d\sigma ds.
    \end{equation*}
    For simplicity, we restrict attention to solutions forward in time. Fix $T > 0$ and suppose that $\Xi_1$ is bounded $H^s \to C_t H_x^s([0,T))$ for some $s$ in the sense that
    \begin{equation}\label{analytic/supposition}
        \big\|\Xi(\phi)\big\|_{L_t^\infty H^{s}([0,T))} \lesssim \|\phi\|_{H^{s}}^{p+1}.
    \end{equation}
    We aim to show that this supposition implies $s \geq s_m$.

    Fix $\widehat{\phi}$ to be a smooth bump function supported on the annulus $\{\frac{1}{4} \leq |\xi| \leq 4\}$ such that $\widehat{\phi} = 1$ on the annulus $\{\frac{1}{2} \leq |\xi| \leq 2\}$. For $N \gg 1$, define 
    \begin{equation*}
        \psi_N = N^d e^{i(1 \wedge T)\Delta}\phi(Nx),
    \end{equation*} 
    so that $|\widehat{\psi_N}| = \widehat{\phi}(\xi/N)$ is supported on frequencies like $N$. The propagator $e^{i(1 \wedge T)\Delta}$ will be used in the progression from \eqref{analytic/1} to \eqref{analytic/2}. In the calculations to follow, we suppress the dependence of constants on the exact profile of $\phi$.

    Using the dual formulation of $H^s$ and testing $\Xi_1(\psi_N)$ against $e^{it\Delta} \psi_N$, we may estimate $\Xi_1(\psi_N)$ from below as
    \begin{align*}
        \big\|\Xi_1(&\psi_N)\big\|_{L_t^\infty H_x^s ([0,T))} \\
        & \gtrsim \bigg\|\frac{1}{\|e^{it\Delta}\psi_N\|_{H_x^{-s}}} \int_0^t \int_0^1 \int_{\R^d} \overline{e^{it\Delta}\psi_N}\cdot e^{i(t-s-\sigma)\Delta} \big[|e^{i(s+\sigma)\Delta}\psi_N|^p\cdot e^{i(s+\sigma)\Delta} \psi_N \big]dx d\sigma ds \bigg\|_{L_t^\infty([0,T))} \\
        % & = \bigg\|\frac{1}{\|\psi_N\|_{H_x^{-s}}} \int_0^t \int_0^1 \int_{\R^d} \big|e^{i(s+\sigma)\Delta}\psi_N\big|^{p+2}dx d\sigma ds \bigg\|_{L_t^\infty([0,T))} \\
        & = \frac{1}{\|\psi_N\|_{H_x^{-s}}} \int_0^T \int_0^1 \int_{\R^d} \big|e^{i(s+\sigma)\Delta}\psi_N\big|^{p+2}dx d\sigma ds.
    \end{align*}
    
    We make the judicious change of variables $\tau = \sigma + s$ and $\rho = s - \sigma$. Noting the inclusion 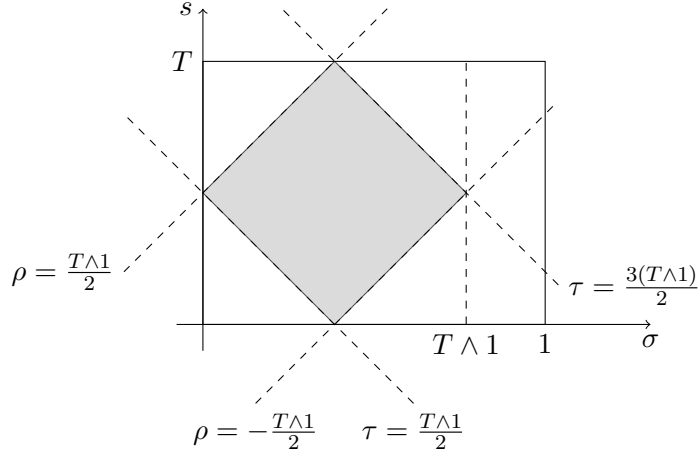
\begin{figure}
        \centering
        \begin{tikzpicture}[scale=3.5]
        % Axes
        \draw[->] (-0.1,0) -- (1.7,0) node[below] {$\sigma$};
        \draw[->] (0,-0.1) -- (0,1.2) node[left] {$s$};
        % Outer rectangle: {0 <= sigma <= 1} x {0 <= s <= T}
        \draw (0,0) rectangle (1.3,1);
        % \node at (0.5,-0.15) {$0 \leq \sigma \leq 1$};
        % \node[rotate=90] at (-0.25,0.5) {$0 \leq s \leq T$};
        \node[left] at (0,1) {$T$};
        \node[below] at (1.3,0) {$1$};
        \draw[dashed] (1,0) node[below] {$T \wedge 1$} -- (1,1);
        % Parallelogram (inner region): image of (gamma,tau) rectangle
        \coordinate (A) at (0.5,0);   % ( (T∧1)/2 , 0 )
        \coordinate (B) at (0,0.5);   % ( 0 , (T∧1)/2 )
        \coordinate (C) at (1,0.5);   % ( T∧1 , (T∧1)/2 )
        \coordinate (D) at (0.5,1);   % ( (T∧1)/2 , T∧1 )
        \draw[dashed] ($(A) + (-0.3,-0.3)$) node[below] {$\rho = -\frac{T\wedge 1}{2}$} -- ($(C) + (0.35,0.35)$);
        \draw[dashed] ($(B) + (-0.3,-0.3)$) node[left] {$\rho =\frac{T\wedge 1}{2}$} -- ($(D) + (0.2,0.2)$);
        \draw[dashed] ($(A) + (0.3,-0.3)$) node[below] {$\tau=\frac{T\wedge 1}{2}$} -- ($(B) + (-0.3,0.3)$);
        \draw[dashed] ($(C) + (0.35,-0.35)$) node[right] {$\tau=\frac{3(T\wedge 1)}{2}$} -- ($(D) + (-0.2,0.2)$);
        \draw[fill=gray!40,opacity=0.7] (A)--(B)--(D)--(C)--cycle;
        \end{tikzpicture}
        \caption{The change of variables $\tau = \sigma + s$ and $\rho = s - \sigma$ in the case $T < 1$.}
        \label{analytic/change of variables}
    \end{figure}
    \begin{equation*}
        \big\{-\tfrac{T\wedge 1}{2} \leq \rho \leq \tfrac{T \wedge 1}{2}\big\} \times \big\{\tfrac{T\wedge 1}{2} \leq \tau \leq \tfrac{3( T \wedge 1)}{2}\big\} \subset \{0 \leq \sigma \leq 1\} \times \{0 \leq s \leq T\},
    \end{equation*}
    illustrated in Figure \ref{analytic/change of variables} for $T < 1$, we may then estimate
    \begin{align}
        \nonumber \big\|\Xi_1(\psi_N)\big\|_{L_t^\infty H_x^s ([0,T))}
        & \gtrsim \frac{1}{\|\psi_N\|_{H_x^{-s}}} \int_\frac{T\wedge 1}{2}^\frac{3( T \wedge 1)}{2} \int_{-\frac{T\wedge 1}{2}}^{\frac{T \wedge 1}{2}} \int_{\R^d} \big|e^{i\tau\Delta}\psi_N\big|^{p+2}dx d\rho d\tau \\
        % \nonumber & = \frac{T\wedge 1}{\|\psi_N\|_{H_x^{-s}}} \int_\frac{T\wedge 1}{2}^\frac{3( T \wedge 1)}{2}\int_{\R^d} \big|e^{i\tau\Delta}\psi_N\big|^{p+2}dx d\tau \\
        \label{analytic/1} & = \frac{T\wedge 1}{\|\psi_N\|_{H_x^{-s}}} \int_\frac{T\wedge 1}{2}^\frac{3( T \wedge 1)}{2}\int_{\R^d} \big|e^{i(\tau + T \wedge 1)\Delta}N^d\phi(Nx)\big|^{p+2}dx d\tau \\
        \label{analytic/2} & = \frac{T\wedge 1}{\|\psi_N\|_{H_x^{-s}}} \int_{-\frac{T\wedge 1}{2}}^\frac{T \wedge 1}{2}\int_{\R^d} \big|e^{i\tau\Delta}N^d\phi(Nx)\big|^{p+2}dx d\tau.
    \end{align}
    Here we note that the shift $e^{i(T \wedge 1)\Delta}$ on $\psi_N$ was used to center the integral over $\tau$ at $\tau = 0$. This is necessary to ensure that the region of integration doesn't escape to infinity as $N \to \infty$ once we make our next change of variables.

    Making the change of variables $y = Nx$ and $r = N^2 \tau$, we find that
    \begin{align*}
        \big\|\Xi_1(\psi_N)\big\|_{L_t^\infty H_x^s ([0,T))}
        & \gtrsim \frac{(T\wedge 1)N^{d(p+2) - d - 2}}{\|\psi_N\|_{H_x^{-s}}} \int_{-\frac{N^2(T\wedge 1)}{2}}^\frac{N^2(T \wedge 1)}{2}\int_{\R^d} \big|e^{ir\Delta_y}\phi(y)\big|^{p+2}dy dr.
    \end{align*}
    Assuming $N^2 \geq \frac{2}{T\wedge 1}$, we find 
    \begin{equation}\label{analytic/3}\begin{split}
        \big\|\Xi_1(\psi_N)\big\|_{L_t^\infty H_x^s ([0,T))}
        & \gtrsim \frac{(T\wedge 1)N^{d(p+2) - d - 2}}{\|\psi_N\|_{H_x^{-s}}} \int_{-1}^1\int_{\R^d} \big|e^{ir\Delta_y}\phi(y)\big|^{p+2}dy dr. 
        % & \gtrsim \frac{(T\wedge 1)N^{d(p+2) - d - 2}}{\|\psi_N\|_{H_x^{-s}}}.
    \end{split}\end{equation}

    We now consider the $H^\alpha$ norm of $\psi_N$. As $\widehat{\psi_N}$ is supported on frequencies like $N$, direct calculation implies that for $N \gg 1$ and all $\alpha \in\R$,
    \begin{equation*}
        \|\psi_N\|_{H^\alpha} \sim N^{\alpha + \frac{d}{2}}.
    \end{equation*}
    Combining our supposition \eqref{analytic/supposition} with the lower bound \eqref{analytic/3}, this then implies that
    \begin{align*}
        N^{d(p+2) - d - 2 + s - \frac{d}{2}} & \gtrsim_T  N^{(p+1)(s + \frac{d}{2})} \\
        1 & \gtrsim_T N^{p(\frac{d}{2} - \frac{2}{p} - s) }.
    \end{align*}
    For any fixed $T$, taking $N \to \infty$ then concludes that $s \geq s_m = \frac{d}{2} - \frac{2}{p}$, as desired.
\end{proof}

% \section{Almost conservation law}
%     \input{I}

\section{Mass sub-critical norm inflation}\label{negative}

In this section, we prove Theorem~\ref{intro/inflation-negative} and show norm inflation in $H^s$ for the cubic one-dimensional \eqref{GT} with $\jb{\gamma} = -1$ for $s < s_i = -\frac{3}{2}$. 
Our approach is based on the work of Kishimoto \cite{series-Kishimoto} and Oh \cite{series-Oh} on the monomial \eqref{NLS}, with refinements to account for the nonlocal structure of \eqref{GT}.

Our restriction to the cubic one-dimensional model is largely for clarity of exposition, though it has the added benefit of covering the most applicable case. We expect that this method extends naturally to other dimensions and nonlinearities, provided $s < \min(s_i,0)$.

Throughout this section, we fix some $R > 0$ and large parameters $1 \ll A \ll N$. We define our initial data $\phi = \phi(A,N,R; x)$ as follows:
\begin{equation}\label{negative/data}
    \widehat{\phi} = R\big(\1_{[N,N+A]} + \1_{[2N, 2N + A]}\big).
\end{equation}
In the calculations to follow, we work explicitly with $\phi$ and suppress the dependence of $\phi$ on $A,N,R$. Regardless, all implied constants will be independent of $A,N,R$.

\subsection{Multilinear estimates}
The proof of Theorem \ref{intro/inflation-negative} is based on two lemmas, which we establish in this subsection. First, in Lemma~\ref{negative/lower} we find a lower bound for the first nonlinear term $\Xi_1$ in \eqref{lwp/power-series}. Second, in Lemma~\ref{negative/upper} we find a upper bound for all higher-order terms $\Xi_j$. These will be combined in the following subsection these to prove Theorem~\ref{intro/inflation-negative}.

In this subsection, all estimates are valid for $s < -\frac{d}{2} = -\frac{1}{2}$. It is only in the ultimate proof of Theorem \ref{intro/inflation-negative} that the requirement $s < -\frac{3}{2}$ becomes necessary. 
To extend this argument to the case $-\frac{d}{2} < s < s_i$ for general $p$ and $d$, some adaptations are needed to account for the fact that $\jb{\xi}^s \notin L^2$; see \cite{series-Oh}.

\begin{lemma}\label{negative/lower}
Fix $s < -\frac{1}{2}$, $R > 0$, $1 \ll A \ll N$, and $\phi$ as in \eqref{negative/data}. Then
\begin{equation*}
    \big\|[\Xi_1(\phi)](t)\big\|_{H^s} \gtrsim t N^{-2} A^2 R^3,
\end{equation*}
uniformly in $A,N,R$ and $0 < t \ll N^{-2}$. 
\end{lemma}
\begin{proof}
    By the definition of $\Xi_1$ and the Plancherel
    theorem, we may write
    \begin{equation*}
        \widehat{[\Xi_1(\phi)]}(t,\xi) = ie^{it\xi^2} \int_0^t \int_0^1 \int_{\xi = \xi_1 - \xi_2 + \xi_3} e^{-i(\tau + \sigma)(\xi^2 - \xi_1^2 + \xi_2^2 - \xi_3^2)} \widehat{\phi}(\xi_1)\overline{\widehat{\phi}(\xi_2)}\widehat{\phi}(\xi_3)d\xi_1 d\xi_2 d\sigma d\tau.
    \end{equation*}
    In all calculations to follow, unless otherwise stated, we restrict our attention to low frequencies $|\xi| \leq A$. In this case, because $A \ll N$, we must have $\xi_1,\xi_3 \in [N, N+A]$ and $\xi_2 \in [2N,2N+A]$. In particular, this implies that $\xi^2 - \xi_1^2 + \xi_2^2 - \xi_3^2 \sim N^2$.

    We focus first on the integral over $\tau$. As $t \ll N^{-2}$ and $\xi^2 - \xi_1^2 + \xi_2^2 - \xi_3^2 \sim N^2$, we find that
    \begin{equation*}
        \bigg|\int_0^t e^{-i\tau(\xi^2 - \xi_1^2 + \xi_2^2 - \xi_3^2)} d\tau\bigg| \gtrsim t.
    \end{equation*}
    This then implies that we may estimate
    \begin{align*}
        \Big|\widehat{[\Xi_1(\phi)]}(t,\xi)\Big|
        & \gtrsim t \bigg|\int_{\xi = \xi_1 - \xi_2 + \xi_3} \int_0^1 e^{-i\sigma(\xi^2 - \xi_1^2 + \xi_2^2 - \xi_3^2)}d\sigma \widehat{\phi}(\xi_1)\overline{\widehat{\phi}(\xi_2)}\widehat{\phi}(\xi_3)d\sigma d\xi_1 d\xi_2 \bigg|.
    \end{align*}
    Evaluating the innermost integral over $\sigma$ explicitly, we find that
    \begin{align*}
        \Big|\widehat{[\Xi_1(\phi)]}(t,\xi)\Big|
        % & \gtrsim t \bigg|\int_{\xi = \xi_1 - \xi_2 + \xi_3} \int_0^1 e^{-i\sigma(\xi^2 - \xi_1^2 + \xi_2^2 - \xi_3^2)}d\sigma \widehat{\phi}(\xi_1)\overline{\widehat{\phi}(\xi_2)}\widehat{\phi}(\xi_3)d\xi_1 d\xi_2 \bigg|\\
        & \gtrsim t \int_{\xi = \xi_1 - \xi_2 + \xi_3} \frac{1}{|\xi^2 - \xi_1^2 + \xi_2^2 - \xi_3^2|} \widehat{\phi}(\xi_1)\overline{\widehat{\phi}(\xi_2)}\widehat{\phi}(\xi_3)d\xi_1 d\xi_2 \\
        & \hspace{10pt} - t \bigg|\int_{\xi = \xi_1 - \xi_2 + \xi_3} \frac{e^{-i(\xi^2 - \xi_1^2 + \xi_2^2 - \xi_3^2)}}{\xi^2 - \xi_1^2 + \xi_2^2 - \xi_3^2}\widehat{\phi}(\xi_1)\overline{\widehat{\phi}(\xi_2)}\widehat{\phi}(\xi_3)d\xi_1 d\xi_2 \bigg| \\
        & = \rn{1} - \rn{2}.
    \end{align*}

    First consider the term $\rn{1}$. Using the fact that $\xi^2 - \xi_1^2 + \xi_2^2 - \xi_3^2 \sim N^2$, we may estimate
    \begin{equation*}
        \rn{1} \gtrsim tN^{-2} \int_{\xi = \xi_1 - \xi_2 + \xi_3}\widehat{\phi}(\xi_1)\overline{\widehat{\phi}(\xi_2)}\widehat{\phi}(\xi_3)d\xi_1 d\xi_2.
    \end{equation*}
    As this is a three-fold convolution of indicator functions of intervals of width $A$, basic convolution properties imply that
    \begin{equation*}
        \rn{1} \gtrsim tN^{-2} A^2 R^3,
    \end{equation*}
    for $|\xi|\leq A$. This is the desired lower bound for $\Xi_1$. 
    It then suffices to show that $\rn{2}$ does not negate this quantity.

    We now focus our attention on term $\rn{2}$. Informally, because $e^{-i(\xi^2 - \xi_1^2 + \xi_2^2 - \xi_3^2)}$ is oscillating, we expect large cancellations and for $\rn{2} \ll \rn{1}$. To make this precise, we treat $\rn{2}$ as an oscillatory integral.
    
    By the definition of $\xi_3$, we may factor
    \begin{equation*}
        \xi^2 - \xi_1^2 + \xi_2^2 - \xi_3^2 = -2(\xi_1 - \xi_2)(\xi_1 - \xi).
    \end{equation*}
    Then by the definition \eqref{negative/data} of $\phi$, we may rewrite
    \begin{align*}
        \rn{2} = tR^3\bigg|\int_N^{N+A} \int_{2N}^{2N + A} \frac{e^{2i(\xi_1 - \xi_2)(\xi_1 - \xi)}}{2(\xi_1 - \xi_2)(\xi_1 - \xi)}\1_{[N,N+A]}(\xi - \xi_1 + \xi_2)d\xi_2 d\xi_1\bigg|.
    \end{align*}
    We focus on the integral over $\xi_2$ as this will generate our cancellations. The indicator function $\1_{[N,N+A]}(\xi_3)$ will restrict the region of integration to a smaller interval $[a,b] \subset [2N, 2N+A]$ which depends on the output frequency $\xi$. For our purposes, the precise formula for $[a,b]$ is irrelevant, so we leave it undefined.
    
    By Van der Corput's Lemma~(e.g., \cite[Cor.~2.6.8]{grafakos}), we may estimate
    \begin{align*}
        \rn{2} 
        & \lesssim tR^3\int_N^{N+A} \frac{1}{|\xi_1 - \xi|} \bigg[ \frac{1}{|\xi_1 - b| |\xi_1 - \xi|} + \int_a^b \frac{1}{|\xi_1 - \xi| |\xi_1-\xi_2|^2} d\xi_2\bigg]d\xi_1 \\
        & \lesssim tR^3AN^{-1}[ N^{-2} + AN^{-3}] \\
        & \lesssim tR^3AN^{-3},
    \end{align*}
    where we note that $|\xi_1 - \xi|, |\xi_1 - b|,|\xi_1 - \xi_2| \gtrsim N$ uniformly for $|\xi| \leq A$.

    Combining the estimates for $\rn{1}$ and $\rn{2}$ and choosing $N \gg 1$, we then find that
    \begin{equation*}
        \Big|\widehat{[\Xi_1(\phi)]}(t,\xi)\Big| \gtrsim tN^-2 A^2 R^3(1 - A^{-1}N^{-2}) \gtrsim tN^{-2} A^2 R^3.
    \end{equation*}
    As $A \gg 1$ and $\jb{\xi}^s \in L^2$ because $s < -\frac{1}{2}$, direct calculation then implies that
    \begin{equation*}
        \|\Xi_1(\phi)\|_{H^s} \gtrsim tN^{-2} A^2 R^3 \|\jb{\xi}^s\|_{L_\xi^2(|\xi|\leq A)} \geq tN^{-2} A^2 R^3,
    \end{equation*}
    as desired. This concludes the proof of Lemma \ref{negative/lower}.
\end{proof}

We now show upper bounds for the higher order terms $\Xi_j$ in \eqref{lwp/power-series}. Compared to prior work on the monomial \eqref{NLS} \cite{series-Kishimoto,series-Bourgain}, here we must exploit the nonlocal structure of \eqref{GT} to find additional cancellations.
\begin{lemma}\label{negative/upper}
Fix $s < -\frac{1}{2}$, $R > 0$, $1 \ll A \ll N$, and $\phi$ as in \eqref{negative/data}. Then for all $j$,
\begin{equation*}
    \|[\Xi_j(\phi)](t)\|_{H^s} \leq C^j t^j R^{2j+1} (\log A)^{2j},
\end{equation*}
for some universal constant $C$ which is uniform in $R,N,A$ and $0 < t \ll N^{-2}$. 
\end{lemma}
\begin{proof}
    For the sake of induction, we prove the stronger, frequency-wise bound
    \begin{equation}\label{negative/upper/claim}
        \Big|\widehat{[\Xi_j(\phi)]}(t,\xi)\Big| \leq t^j R^{2j + 1} (\log A)^{2j} \frac{C^j}{(1+j)^2} \sum_{I \in \mathcal{I}_j} \1_I(\xi),
    \end{equation}
    where $C$ is a large universal constant and $\mathcal{I}_j$ is a collection of at most $\frac{2C^j}{(1+j)^2}$ intervals of width $A$. We note that the factors of $(1+j)^{-2}$ are solely a combinatorial trick to close the induction; they serve no other purpose in the analysis. Because $s < -\frac{1}{2}$ and $\jb{\xi}^{2s}$ is integrable, \eqref{negative/upper/claim} is sufficient to conclude the proof of the lemma.

    We proceed by induction. By construction, \eqref{negative/upper/claim} holds for $\Xi_0$, so it suffices to show the inductive step. Suppose for the sake of induction that \eqref{negative/upper/claim} holds for all $k < j$. 
    In the following calculations, we suppress the dependence on $\phi$; i.e. we let $\widehat{\Xi_k}(t,\xi) = \widehat{[\Xi_k(\phi)]}(t,\xi)$.
    
    By definition, we may write
    \begin{align*}
        \widehat{\Xi_j} (t,\xi) = i\sum_{k + \ell + m = j - 1} \int_0^t\int_{\xi = \xi_k - \xi_\ell + \xi_m} e^{-i(t-s)\xi^2} \int_0^1 &  e^{i\sigma(\xi^2 - \xi_k^2 + \xi_\ell^2 - \xi_m^2)} d\sigma \\
        & \cdot \Big[\widehat{\Xi}_k(s,\xi_k) \overline{\widehat{\Xi}_\ell(s,\xi_\ell)} \widehat{\Xi}_m(s,\xi_m) \Big]d\xi_k d\xi_\ell ds.
    \end{align*}
    We note that we may bound the innermost integral over $\sigma$ by $1$. Alternatively, we may evaluate the innermost integral over $\sigma$ explicitly to find an upper bound which depends inversely on $\xi^2 - \xi_k^2 + \xi_\ell^2 - \xi_m^2$. Choosing the minimum of these quantities and invoking the inductive hypothesis, we may then estimate $\Xi_j$ as
    \begin{align*}
        \big|\widehat{\Xi_j} (t,\xi)\big| \leq \sum_{k + \ell + m = j - 1}\sum_{K \in \mathcal{I}_k} \sum_{L \in \mathcal{I}_\ell}\sum_{M \in \mathcal{I}_m} & R^{2j+1} (\log A)^{2j-2} \tfrac{C^{j-1}}{(1+k)^2 (1+\ell)^2 (1+m)^2} \\ 
        & \cdot\int_K \int_L\int_0^t s^{j-1} \Big[1 \wedge \tfrac{2}{|\xi^2 - \xi_k^2 + \xi_\ell^2 - \xi_m^2|} \Big] \1_M(\xi - \xi_k + \xi_\ell) ds d\xi_\ell d\xi_k.
    \end{align*}
    
    Integrating in $s$, factoring $|\xi^2 - \xi_k^2 + \xi_\ell^2 - \xi_m^2| = 2 |\xi -\xi_k||\xi_k - \xi_\ell|$, and noting that $\1_{M}(\xi - \xi_k + \xi_\ell) \leq \1_{K - L + M}(\xi)$, we find that
    \begin{align*}
        \big|\widehat{\Xi_j} (t,\xi)\big| \leq \sum_{k + \ell + m = j - 1} \sum_{K \in \mathcal{I}_k} \sum_{L \in \mathcal{I}_\ell}\sum_{M \in \mathcal{I}_m} & R^{2j+1} (\log A)^{2j-2} t^j \tfrac{C^{j-1}}{(1+k)^2 (1+\ell)^2 (1+m)^2} \\ 
        & \cdot \1_{K - L + M}(\xi) \int_K \int_L \Big[1 \wedge \tfrac{1}{|\xi -\xi_k||\xi_k - \xi_\ell|} \Big] d\xi_\ell d\xi_k.
    \end{align*}
    As $K$ and $L$ are of width $A$, integrating over $\xi_\ell$ or $\xi_k$ yields, at worst, $1 + \log A$. Noting that $(1 + \log A)^2 \leq 2 (\log A)^2$ for $A \gg 1$, we then find
    \begin{align*}
        \big|\widehat{\Xi_j} (t,\xi)\big| \leq & R^{2j+1} (\log A)^{2j} t^j\sum_{k + \ell + m = j - 1} \tfrac{2C^{j-1}}{(1+k)^2 (1+\ell)^2 (1+m)^2} \sum_{K \in \mathcal{I}_k} \sum_{L \in \mathcal{I}_\ell}\sum_{M \in \mathcal{I}_m}  \1_{K - L + M}(\xi).
    \end{align*}

    We now invoke the combinatorial trick. As $k + \ell + m = j - 1$, we know that $1+j$ satisfies 
    \begin{equation*}
        1 + j \leq 3 \big[(1+ k) \vee (1 + \ell) \vee (1+m)\big].
    \end{equation*}
    This then implies that
    \begin{align} \label{negative/5}
        \sum_{k + \ell + m = j - 1} \frac{2C^{j-1}}{(1+k)^2 (1+\ell)^2 (1+m)^2} & \leq \frac{18 C^{j-1}}{(1+j)^2} \bigg( \sum_{n \geq 0} \frac{1}{(1+n)^2}\bigg)^2.
    \end{align}
    Choosing $C \geq 18 ( \sum_{n \geq 0} (1+n)^{-2})^2$, we find
    \begin{align}\label{negative/6}
        \sum_{k + \ell + m = j - 1} \frac{2C^{j-1}}{(1+k)^2 (1+\ell)^2 (1+m)^2} & \leq \frac{C^j}{(1+j)^2},
    \end{align}
    as desired.

    For the sum over intervals, we invoke the same trick. We note that the set $K - L + M$ may be constructed of three intervals of width $A$. We may then bound
    \begin{align*}
        \sum_{k + \ell + m = j -1}\sum_{K \in \mathcal{I}_k} \sum_{L \in \mathcal{I}_\ell}\sum_{M \in \mathcal{I}_m}  \1_{K - L + M}(\xi) \leq \sum_{k + \ell + m = j - 1}\sum_{I \in \tilde{\mathcal{I}_j}} \1_I(\xi),
    \end{align*}
    where $\tilde{\mathcal{I}}_j$ is a collection of, at most, $\frac{24 C^{j-1}}{(1+k)^2(1+\ell)^2 (1+m)^2}$ intervals of width $A$.
    By the same calculation as previously, we may choose 
    \begin{equation*}
        2C \geq 24\cdot 3^2 \cdot \bigg(\sum_{n\geq 0}\tfrac{1}{(1+n)^{2}}\bigg)^2
    \end{equation*} 
    to conclude \eqref{negative/upper/claim} and complete the proof of Lemma~\ref{negative/upper}.
\end{proof}

\subsection{Proof of Theorem~\ref{intro/inflation-negative}}
In this subsection, we prove Theorem~\ref{intro/inflation-negative} and show norm inflation for the one-dimensional cubic \eqref{GT} with $\jb{\gamma} = -1$.

\begin{proof}
    Fix our initial data $\phi$ as defined in \eqref{negative/data} for $N, A, R \gg 1$ to be chosen later and let $T$ denote the time at which norm inflation occurs, to be determined later. From our preceding lemmas, a number of relationships are necessary between the parameters $A, N, R, T$:
    \begin{enumerate}[(i)]
        \item \label{requirement/lwp}
            {\em Local well-posedness:}
            Due to the local well-posedness theory in Theorem~\ref{intro/lwp}, we require
            \begin{equation}\label{negative/7}
                T\|\phi\|_{L^2}^2 \ll 1 \quad \iff \quad T R^2 A \ll 1.
            \end{equation}
            This ensures that a solution to \eqref{GT} with initial data $\phi$ exists and can be expressed by the power series expansion \eqref{lwp/power-series} on the interval $(-T,T)$.\footnote{Curiously, this requirement is the ultimate limiter for the values of $s$ at which we find norm inflation.}
        \item\label{requirement/small-data}
            {\em Small initial data:}
            To ensure that the initial data $\phi$ is small, we require that 
            \begin{equation*}
                \|\phi\|_{H^s} \sim N^s R A^{\frac{1}{2}} \ll 1.
            \end{equation*}
        \item\label{requirement/growth}
            {\em Norm growth:}
            The norm growth of the solution will be driven by a high-to-low frequency cascade in the first nonlinear term $\Xi_1$ in \eqref{lwp/power-series}. From Lemma~\ref{negative/lower}, this implies that we need
            \begin{equation*}
                T N^{-2} A^2 R^3 \gg 1.
            \end{equation*}
        \item\label{requirement/convergence}
            {\em Convergence of higher order terms:} To ensure that the growth of $\Xi_1$ is not negated by higher order effects, we require the higher order terms $\Xi_j$ to converge and be dominated by $\Xi_1$. Comparing the upper bounds from Lemma~\ref{negative/upper} to the lower bound in Lemma~\ref{negative/lower}, this corresponds to two requirements:
            \begin{equation*}
                TR^2 (\log A)^2 \ll 1 \qtq{and} TR^2 N^2 A^{-2} (\log A)^4 \ll 1,
            \end{equation*}
            the first ensuring that the higher order terms converge and the second that $\Xi_1$ dominates.
        
        \item\label{requirement/separation}
            {\em Separation:} From our analysis in Lemmas \ref{negative/lower} and \ref{negative/upper}, we require
            \begin{equation*}
                1 \ll A \ll N.
            \end{equation*}
            This ensures that $\widehat{\phi}$ consists of two disjoint indicator functions.

        \item\label{requirement/instantaneity}
            {\em Instantaneity:} In order to force norm inflation to occur instantly, we require
            \begin{equation*}
                T \ll 1.
            \end{equation*}

    \end{enumerate}
    %If this requirement can be relaxed by a better local well-posedness theory, then norm inflation can be achieved for a wider range of $s$.}
    
    \noindent To achieve these requirements, we fix some $N \gg 1$ and then choose\footnote{We note that this differs from the parameters one would choose for NLS. The time $T$ is significant smaller and $R$ is significantly larger; see \cite{series-Oh}.}
    \begin{equation*}
        R = N^{1 + 3\delta}, \quad A = N^{1-\delta}, \qtq{and} T = N^{-3 - 6 \delta},
    \end{equation*}
    where $\delta > 0$ is chosen sufficiently small so that $s + \frac{3}{2} + \frac{5}{2} \delta < 0$. This satisfies requirement \eqref{requirement/instantaneity}. With this choice, we find that
    \begin{equation*}
        T \|\phi\|_{L^2}^2 = N^{-\delta} \ll 1,
    \end{equation*}
    satisfying requirement \eqref{requirement/lwp}. In addition, we find that
    \begin{equation*}
        \|\phi\|_{H^s} \sim N^s R A^{\frac{1}{2}} = N^{s + \frac{3}{2} + \frac{5}{2} \delta} = N^{0-},
    \end{equation*}
    satisfying requirement \eqref{requirement/small-data}. Invoking Theorem~\ref{intro/lwp}, let $u : (-T,T) \to \C$ be the solution to \eqref{GT} with initial data $u_0 = \phi$ given by the power series \eqref{lwp/power-series}.
    
    Choosing $N \gg 1$ sufficiently large, we may ensure that $(\log A)^2 \leq N^{\delta}$. At time $T$, requirement \eqref{requirement/separation}, Lemma \ref{negative/lower}, and Lemma \ref{negative/upper} then imply that we may estimate
    \begin{align*}
        \|u(T)\|_{H^s} 
        & \geq \|[\Xi_1(\phi)](T)\|_{H^s} - \|\phi\|_{H^s} - \sum_{j \geq 2} \|[\Xi_j(\phi)](T)\|_{H^s} \\
        & \sim TN^{-2} A^2 R^3 - N^s R A^{\frac{1}{2}} - \sum_{j \geq 2} C^j T^j R^{2j+1} (\log A)^{2j} \\
        % & \geq N^\delta - N^{0-} -  \sum_{j \geq 2} C^j N^{-1 + 5\delta}(CN^{-1 + \delta})^{j-2} \\
        & \geq N^\delta - N^{0-} - C^2 N^{-1 + 5\delta} \sum_{j \geq 0} (CN^{-1 + \delta})^j \\
        & \gtrsim N^\delta,
    \end{align*}
    wherein we see that requirements \eqref{requirement/growth} and \eqref{requirement/convergence} were used.
    Taking $N \to \infty$, we find that $\|\phi\|_{H^s} \to 0$, $\|u(T)\|_{H^s} \to \infty$ and $T \to 0$ which concludes the proof of Theorem~\ref{intro/inflation-negative}.
\end{proof}

\section{Energy super-critical norm inflation}
    In this section, we prove Theorem \ref{intro/inflation-energy} and show norm inflation in $H^s$ for \eqref{GT} provided $1 \leq s < s_i$. This is built on a virial identity, which we show in Subsection~\ref{virial}, and the resulting energy equipartition phenomenon, which we show in Subsection~\ref{energy-equipartition}.
    
        \subsection{Virial identity}\label{virial}
            In this section, we extend the virial identity to the general \eqref{GT} equation; originally due to Glassey \cite{NLS-virial-glassey} for \eqref{NLS} and shown for the focusing one dimensional \eqref{GT} by Choi, Hong, and Lee \cite{GT-dichotomy}. For comparison to \cite{GT-dichotomy}, we note that our energy is the negative of the energy used by \cite{GT-dichotomy}. This is chosen for the sake of notation and so that the $E(u) \geq 0$ in the defocusing, $\jb{\gamma} < 0$ case. 

As this virial identity is only used for Schwartz functions, we do not concern ourselves with questions of regularity or decay. Instead, we assume that $u$ is sufficiently regular and decaying as to justify all techniques used. The necessary regularity and decay are discussed in more detail in \cite{GT-dichotomy}.

\begin{proof}[Proof of Proposition~\ref{intro/virial}]
    Define the variance of $u$ as
    \begin{equation*}
        v(t) = \int |x|^2 |u(t,x)|^2 dx
    \end{equation*}
    When it is convenient, we use subscripts of $u$ to indicate partial derivatives and subscripts of $x$ to indicate coordinates. In the usual way, any repeated indices are summed.
    
    Taking a derivative of $v$ in time and applying integration by parts, we find that
    \begin{align}
        \dot{v}(t) & = 2 \Re \int|x|^2 \overline{u} u_t \nonumber \\
        & = - 2 \Im \int|x|^2\overline{u}\bigg[\jb{\gamma} u_{jj} + \int_0^1 e^{-i\sigma \Delta} \big(|e^{i\sigma \Delta} u|^{p} \cdot e^{i\sigma \Delta} u \big)d\sigma \bigg]dx \nonumber\\
        % & =  4 \jb{\gamma} \Im \intx_j\overline{u} u_{j} dx \\
        % & \hspace{10pt} + 2 \jb{\gamma} \Im \int|x|^2\overline{u_j} u_{j} dx \\
        % & \hspace{10pt} - 2 \Im \iint_0^1  \overline{e^{i\sigma \Delta}(|x|^2u)} \cdot |e^{i\sigma \Delta} u|^{p} \cdot e^{i\sigma \Delta} u \; d\sigma dx\\
        & =  4 \jb{\gamma} \Im \int\overline{u} x_j u_j dx  - 2 \Im \iint_0^1  \overline{e^{i\sigma \Delta}(|x|^2u)} \cdot |e^{i\sigma \Delta} u|^{p} \cdot e^{i\sigma \Delta} u \; d\sigma dx \nonumber\\
        & = \jb{\gamma} \dot{v}_1(t) + \dot{v}_2(t).\label{virial/0}
    \end{align}
    
    We first consider $\dot{v}_2$ and note the identity,
    \begin{align*}
        e^{i \sigma \Delta} |x|^2 
        & = (|x|^2 + 4 i \sigma x \cdot \nabla + 2d i \sigma - 4 \sigma^2 \Delta) e^{i\sigma \Delta}\\
        & = \big[|x|^2 + i(4 \sigma x \cdot \nabla + 2d \sigma + 4 \sigma^2 \partial_\sigma)\big] e^{i\sigma \Delta}.
    \end{align*}
    In the integral, the $|x|^2$ term will vanish due to the imaginary part in $\dot{v}_2$. Integrating by parts again, we then find that
    \begin{align}
        \dot{v}_2(t) 
        & = -2 \Re \iint_0^1 \overline{(4\sigma x \cdot \nabla + 2 d \sigma + 4 \sigma^2 \partial_\sigma) e^{i\sigma \Delta} u} \cdot |e^{i\sigma \Delta} u|^{p} \cdot e^{i\sigma \Delta} u \; d\sigma dx \nonumber \\
        & = 2 \iint_0^1 \Big(\tfrac{4}{p+2}\sigma x \cdot \nabla + 2 d \sigma + \tfrac{4 \sigma^2}{p+2} \partial_\sigma\Big) |e^{i\sigma \Delta} u|^{p+2} d\sigma dx \nonumber \\
        % & =  \tfrac{8}{p+2}\iint_0^1\sigma x \cdot \nabla|e^{i\sigma \Delta} u|^{p+2} d\sigma dx \\
        % & \hspace{10pt} + 4 d\iint_0^1 \sigma|e^{i\sigma \Delta} u|^{p+2} d\sigma dx \\
        % & \hspace{10pt} + \tfrac{8}{p+2}\iint_0^1 \tfrac{4 \sigma^2}{p+2} \partial_\sigma |e^{i\sigma \Delta} u|^{p+2} d\sigma dx \\
        % & =  \tfrac{4dp}{p+2}\iint_0^1\sigma |e^{i\sigma \Delta} u|^{p+2} d\sigma dx - \tfrac{16}{p+2}\iint_0^1 \sigma |e^{i\sigma \Delta} u|^{p+2} d\sigma dx + \tfrac{8}{p+2}\int\sigma^2 |e^{i\sigma \Delta} u|^{p+2}\Big|_{\sigma = 0}^{\sigma = 1} dx \\
        & = \tfrac{4(dp - 4)}{p+2}\iint_0^1\sigma |e^{i\sigma \Delta} u|^{p+2} d\sigma dx + \tfrac{8}{p+2}\int|e^{i\Delta} u|^{p+2}dx. \label{virial/1}
    \end{align}
    In the focusing case $\jb{\gamma} > 0$, considered by \cite{GT-dichotomy}, it is sufficient to remove the $e^{i \Delta} u$ term and bound $\dot{v}_2$ above by $0$. However, in the defocusing case $\jb{\gamma} < 0$, we must use this $e^{i\Delta}$ term to cancel a corresponding term in $\ddot{v}_1$.
    
    We now turn attention to $\dot{v}_1$. Differentiating in $t$ and once again integrating by parts, we find
    \begin{align*}
        \ddot{v}_1(t)
        % & = 4 \Im \int\overline{u} x_j u_{tj} dx + 4 \Im \int\overline{u_t} x_ju_j dx\\
        & = - 4 d \Re \int\overline{u} u_t \; dx - 8 \Im \int u_t x_j \overline{u_j} dx\\
        % & = - 4 d \Re \int\overline{u}\bigg[\jb{\gamma} u_{kk} + \int_0^1 e^{-i\sigma \Delta} \big(|e^{i\sigma \Delta} u|^{p} \cdot e^{i\sigma \Delta} u \big)d\sigma \bigg] \; dx \\
        % & \hspace{10pt} - 8 \Re \int\bigg[\jb{\gamma} u_{kk} + \int_0^1 e^{-i\sigma \Delta} \big(|e^{i\sigma \Delta} u|^{p} \cdot e^{i\sigma \Delta} u \big)d\sigma \bigg] x_j \overline{u_j} dx \\
        & = 4 d \jb{\gamma}\int|\nabla u|^2 - 4 d \iint_0^1|e^{i\sigma \Delta} u|^{p+2}d\sigma dx - 8\jb{\gamma} \Re \int u_{kk} x_j \overline{u_j} dx \\
        & \hspace{10pt}  - 8 \Re \iint_0^1 \big(|e^{i\sigma \Delta} u|^{p} \cdot e^{i\sigma \Delta} u \big) \overline{e^{i\sigma \Delta}(x \cdot \nabla u)} d\sigma dx \\
        & = (4 d +8)\jb{\gamma}\int|\nabla u|^2 - 4 d \iint_0^1|e^{i\sigma \Delta} u|^{p+2}d\sigma dx + 8 \jb{\gamma} \Re \int u_{j} \overline{u_{jk}} x_k dx\\
        & \hspace{10pt}  - 8 \Re \iint_0^1 \big(|e^{i\sigma \Delta} u|^{p} \cdot e^{i\sigma \Delta} u \big) \overline{e^{i\sigma \Delta}(x \cdot \nabla u)} d\sigma dx \\
        & = (4 d +8) \jb{\gamma}\int|\nabla u|^2 - 4 d \iint_0^1|e^{i\sigma \Delta} u|^{p+2}d\sigma dx  + \rn{3} + \rn{4}.
    \end{align*}
    Consider the Hessian term $\rn{3}$. Integrating by parts once again, we find that
    \begin{align*}
        \rn{3} & = - 8 d\jb{\gamma} \int |\nabla u|^2 dx - 8 \jb{\gamma} \Re \int u_{jk} \overline{u_{j}} x_k dx = - 8 d\jb{\gamma} \int |\nabla u|^2 dx - \rn{3}.
    \end{align*}
    Upon rearranging, this implies that
     \begin{align*}
          \rn{3} & = - 4 d\jb{\gamma} \int |\nabla u|^2 dx.
     \end{align*}
    
    Now consider term $\rn{4}$. We recall the following identity, as operators acting on the right,
    \begin{equation*}
        e^{i \sigma \Delta} (x \cdot \nabla) = ( x \cdot \nabla + 2 i \sigma \Delta) e^{i \sigma \Delta} = (x \cdot \nabla + 2 \sigma \partial_\sigma) e^{i \sigma \Delta}.
    \end{equation*}
    With the ultimate use of integration by parts, this then  implies that
    \begin{align*}
        \rn{4} 
        & = - 8 \Re \iint_0^1 \big(|e^{i\sigma \Delta} u|^{p} \cdot e^{i\sigma \Delta} u \big) (x \cdot \nabla + 2 \sigma \partial_\sigma) \overline{e^{i \sigma \Delta} u} \; d\sigma dx \\
        & = - \tfrac{8}{p+2} \iint_0^1 (x \cdot \nabla + 2 \sigma \partial_\sigma) |e^{i \sigma \Delta} u|^{p+2} \; d\sigma dx \\
        & = \tfrac{8d + 16}{p+2} \iint_0^1 |e^{i \sigma \Delta} u|^{p+2} \; d\sigma dx - \tfrac{16}{p+2} \int\sigma |e^{i \sigma \Delta} u|^{p+2} \Big|_{\sigma = 0}^{\sigma = 1} dx \\
        & = \tfrac{8d + 16}{p+2} \iint_0^1 |e^{i \sigma \Delta} u|^{p+2} \; d\sigma dx - \tfrac{16}{p+2} \int|e^{i \Delta} u|^{p+2} dx.
    \end{align*}
    
    Combining all terms and recalling the formula for the energy \eqref{energy}, we find
    \begin{align}
        \ddot{v}_1(t)
        & = 8 \jb{\gamma}\int|\nabla u|^2 + \tfrac{4(4 - dp)}{p+2} \iint_0^1 |e^{i \sigma \Delta} u|^{p+2} \; d\sigma dx - \tfrac{16}{p+2} \int|e^{i\Delta} u|^{p+2} dx \nonumber\\
        & = -16 E(u) - \tfrac{4(dp - 8)}{p+2} \iint_0^1 |e^{i\sigma \Delta} u|^{p+2} d\sigma dx - \tfrac{16}{p+2} \int |e^{i \Delta} u|^{p+2} dx. \label{virial/2}
    \end{align}
    
    We now restrict our attention to the defocusing case, $\jb{\gamma} = -1$, for $-\frac{1}{2} \leq t \leq 0$ and $p \geq \frac{8}{d}$. In this regime, we find that
    \begin{align*}
        16 E(u) \leq -\ddot{v}_1 & \leq \big(16 + \tfrac{4(dp-8)}{p+2}\big) E(u) + \tfrac{16}{p+2} \int|e^{i\Delta} u|^{p+2} \qtq{and}
        \dot{v}_2 \geq \tfrac{8}{p+2} \int |e^{i\Delta} u|^{p+2} dx.
    \end{align*}
    In this case, the $e^{i\Delta} u$ term in $-\ddot{v}_1$ has a bad sign. However, we can use the corresponding term in $\dot{v}_2$ to counteract it.
    Returning to the full variance evolution, we then find that
    \begin{align*}
        v(t) 
        % & = v(0) + \int_0^t \dot{v}(s) ds \\
        & = v(0) - \dot{v}_1(0) t + \int_0^t \dot{v}_2(s) ds - \int_0^t \int_0^s \ddot{v}_1(\tau) d\tau ds \\
        & \leq v(0) - \dot{v}_1(0) t + C(p,d) E(u) t^2 + \tfrac{8}{p+2} \iint_0^t |e^{i\Delta} u(s)|^{p+2} ds dx + \tfrac{16}{p+2} \iint_0^t \int_0^s |e^{i\Delta} u(\tau)|^{p+2}d\tau ds dx \\
        & \leq v(0) - \dot{v}_1(0) t + C(p,d) E(u) t^2 - \tfrac{8}{p+2} (1 + 2t)\iint_t^0 |e^{i\Delta} u(\tau)|^{p+2}d\tau ds dx.
    \end{align*}
    Provided $-\frac{1}{2}\leq t \leq 0$, this concludes that
    \begin{equation}\label{virial/v}
        v(t) \leq v(0) - \dot{v}_1(0) t + C(p,d) E(u) t^2,
    \end{equation}
    where $C(p,d) > 0$ for $p \geq \frac{8}{d}$.
    
    Finally, we consider the focusing case, $\jb{\gamma} = 1$, for $t < 0$. In this regime, we find that
    \begin{equation*}
        \ddot{v}_1(t) \leq -16E(u)
        \qtq{and}
        \dot{v}_2 \geq 0.
    \end{equation*}
    This implies that
    \begin{align*}
        v(t) & = v(0) + \dot{v}_1(0) t + \int_0^t \dot{v}_2(s) ds + \int_0^t \int_0^s \ddot{v}_1(\tau) d\tau ds \\
        & \leq v(0) + \dot{v}_1(0) t - 8E(u) t^2,
    \end{align*}
    as desired.

    In addition, in either case, we find that $\ddot{v}_1(t) \leq -16 E(u)$ for $t < 0$ and so
    \begin{equation}\label{virial/dot-v_1}
        \dot{v}_1(t) = \dot{v}_1(0) + \int_0^t \ddot{v}_1(s) ds \geq \dot{v}_1(0) - 16 E(u) t,
    \end{equation}
    which concludes the proof of Proposition~\ref{intro/virial}.
    
\end{proof}
        \subsection{Energy equipartition} \label{energy-equipartition}

With the virial identity established, we demonstrate that suitable solutions to \eqref{GT} undergo {\em energy equipartition}, Proposition~\ref{intro/energy-equipartition}, in the defocusing, $\jb{\gamma} = -1$ case.

\begin{proof}
    Let $u_0 \in H^1$ be a real-valued\footnote{The restriction to real-valued initial data is convenient as it forces $\dot{v}_1(0) = 0$, but is not strictly necessary. Without this restriction, the equalization time $T$ would change.} function with sufficient regularity and decay as to justify the virial identity. Let $u$ be the corresponding solution to \eqref{GT} with initial data $u_0$. 
    
    As $u_0$ is real-valued, we find that $\dot{v}_1(0) = 0$. Applying the identity for $\dot{v}_1$ \eqref{virial/dot-v_1}, we then find that
    \begin{equation*}
        -16 E(u) t \leq 4\Im \int \overline{u} (x \cdot \nabla u(t)) dx \lesssim \sqrt{v(t)}\|u(t)\|_{\dot{H}^1}.
    \end{equation*}
    With the virial identity \eqref{virial/v}, we then find that for $-\frac{1}{2} \leq t \leq 0$.
    \begin{align}\label{inflation/equipartition}
        \|u(t)\|^2_{\dot{H}^1} \gtrsim \frac{E(u)^2 t^2}{v(0) + E(u) t^2}.
    \end{align}
    From this identity, we see that at time $T^2 = v(0)/E(u)$ for $T < 0$,
    \begin{equation*}
        \|u(T)\|^2_{\dot{H}^1} \gtrsim E(u).
    \end{equation*}
    In particular, if 
    \begin{equation*}
        \|u_0\|^2_{\dot{H}^1} \ll\iint_{0}^1 |e^{i\sigma \Delta} u_0|^{p+2} d\sigma dx
    \end{equation*}
    then the kinetic energy must become comparable to the potential energy by time $T$. 
\end{proof}

\subsection{Proof of Theorem~\ref{intro/inflation-energy}}

    We now exploit this virial identity and, in the defocusing case, energy equipartition phenomenon to show norm inflation in $H^s$ for $1 \leq s < s_i$.
\begin{proof}
    Fix some $1 \leq s < s_i$ and define $u_0$ to be the Gaussian initial data
    \begin{equation*}
        u_0(x) = A e^{-|x|^2/4 \sigma^2},
    \end{equation*}
    where we will take $A \gg 1$ and $\sigma \ll 1$. Let $v(t)$ denote the variance of $u$ as in the virial and let $T^2 = v(0)/E(u)$ for $T < 0$ denote the time given by Proposition \ref{energy-equipartition}. 
    Let $u$ be the corresponding solution to \eqref{GT} for some $p > \frac{8}{d-2}$ (i.e., $s_i > 1$) and suppose that $u(t) \in H^s$ exists until time $T$.

    We first focus on the defocusing, $\jb{\gamma} = -1$ case because it represents the more interesting result.
    Following the intuition from energy equipartition, we choose $A,\sigma$ so that $\|u_0\|_{H^s} \ll 1$---hence $\|u_0\|^2_{\dot{H}^1} \ll 1$---and $U(0) \gg 1$. Provided that $|T| \ll 1$, this will cause a rapid increase in kinetic energy, and hence norm inflation in $H^s$.
    
    We now turn to the details. For our particular choice of $u_0$, provided that $\sigma \ll 1$, direct calculation implies that
    \begin{align*}
        \|u_0\|^2_{H^s} \sim A^2 \sigma^{d - 2s}, \quad 
        U(0) \sim A^{p+2} \sigma^{d + 2}, \quad \|u_0\|^2_{\dot{H}^1} \sim A^2 \sigma^{d - 2}, \qtq{and} v(0) \sim A^2 \sigma^{d+2}.
    \end{align*}
    Here we draw the reader's attention to the extra factor of $\sigma^2$ that is present in the potential energy term $U(0)$. This $\sigma^2$ comes from the averaging in time in \eqref{energy} and is not present in the corresponding calculation for the energy of \eqref{NLS}. It is this term that lowers our threshold for norm inflation for \eqref{GT} compared to \eqref{NLS}.
    
    We now choose $A$ and $\sigma$. Fix $1 \leq s < s_i$ and $\epsilon > 0$. Because $s < s_i = \frac{d}{2} - \frac{4}{p}$, we may choose
    \begin{equation*}
        \sigma^{2\big(\frac{d}{2} - \frac{4}{p} - s\big)} = \epsilon^{1 + \frac{4}{sp}}= \epsilon^{0+} \qtq{and} A^p = \epsilon^{-\frac{2}{s}} \sigma^{-4} = \epsilon^{0-},
    \end{equation*}
    which implies that $\sigma \ll 1$ and $A \gg 1$.
    In particular, we have chosen $(A^p \sigma^4)^s = \epsilon^{-2}$.
    
    As $A^p \sigma^4 \gg 1$, we see that $U(0) \gg \|u_0\|^2_{\dot{H}^1}$. In particular, with $T < 0$,
    \begin{equation*}
        E(u) \sim U(0) \qtq{and} T^2 \sim A^{-p}.
    \end{equation*}
    Proposition \eqref{intro/energy-equipartition} and interpolating with the conservation of mass then imply that
    \begin{equation*}
        \|u(T)\|^{2}_{H^s} \gtrsim \|u(T)\|_{H^1}^{2s}\|u\|_{L^2}^{2-2s} \gtrsim E(u)^s M(u)^{1-s} \sim (A^p \sigma^4)^s A^{2}\sigma^{d - 2s}.
    \end{equation*}
    For our choice of $u_0$, provided that $\epsilon \ll 1$, we then find that
    \begin{align*}
        \|u_0\|^2_{H^s} \sim A^2 \sigma^{d - 2s} = \epsilon, \quad
        \|u(T)\|^2_{H^s} \gtrsim \epsilon^{-1}, \qtq{and}
        T^2 \sim \epsilon^{0+}.
    \end{align*}
    Sending $\epsilon \to 0$ then implies norm inflation in $H^s$ and concludes the proof of Theorem~\ref{intro/inflation-energy} in the defocusing case $\jb{\gamma} = -1$.
    
    We now consider the focusing case, $\jb{\gamma} = 1$. Here, we take inspiration from the classical virial argument used to show finite time blow-up; see \cite{GT-dichotomy} for its application to \eqref{GT}. 
    From Proposition~\ref{intro/virial}, following the same reasoning as used for Proposition \ref{energy-equipartition}, we find that for $t < 0$,
    \begin{equation}\label{negative/focusing}
        \|u(t)\|_{\dot{H}^1} \gtrsim \frac{E(u)^2 t^2}{v(0) - E(u)t^2}.
    \end{equation}
    In particular, by time $T^2 = v(0)/E(u)$ for $T < 0$, we find that $\|u(T)\|_{H^1} = \infty$ and hence $\|u(T)\|_{H^s} = \infty$ for all $s \geq 1$.
    
    By the same reasoning as in the defocusing case, provided $1 \leq s < s_i$, for any $\epsilon > 0$, we may choose $A,\sigma$ appropriately so that
    \begin{equation*}
        \|u_0\|^2_{H^s}\sim \epsilon, \quad \|u(T)\|_{H^s}^2 = \infty, \qtq{and} T^2 \sim \epsilon^{0+}.
    \end{equation*}
    This concludes norm inflation in $H^s$ in the focusing, $\jb{\gamma} = 1$ case for $1 \leq s < s_i$.
\end{proof}

\subsection{Brief remark on the inter-critical case}
    In the inter-critical case, $0 \leq s_i \leq 1$, norm inflation and ill-posedness remain unresolved. 
    
    In the inter-critical focusing case, the reasoning presented in the proof of Theorem \ref{intro/inflation-energy}, namely \eqref{negative/focusing}, indicates that \eqref{GT} is ill-posed in $H^s$ for $s < s_i$ in the following sense: For all $\epsilon > 0$ and any $s < s_i$, there exists Schwartz initial data $u_0$ with $\|u_0\|_{H^s} < \epsilon$ such that $u$ blows up---in the sense that $u$ fails to be smooth---within time $|T| < \epsilon$. 

    In the defocusing case, blowup solutions are no longer available. This case will be the subject of future work.
    
\bibliographystyle{abbrv}
\bibliography{references}

\begin{thebibliography}{10}

\bibitem{GT-justification-1}
M.~J. Ablowitz and G.~Biondini.
\newblock Multiscale pulse dynamics in communication systems with strong dispersion management.
\newblock {\em Optics Letters}, 23(21):1668--1670, 1998.

\bibitem{DSP-DM-combo-2023}
M.~Abu-romoh, N.~Costa, Y.~Jaou\"{e}n, A.~Napoli, J.~Pedro, B.~Spinnler, and M.~Yousefi.
\newblock Equalization in dispersion-managed systems using learned digital back-propagation.
\newblock {\em Optics Continuum}, 2(10):2088--2105, Oct 2023.

\bibitem{series-Tao}
I.~Bejenaru and T.~Tao.
\newblock Sharp well-posedness and ill-posedness results for a quadratic non-linear {S}chr\"odinger equation.
\newblock {\em Journal of Functional Analysis}, 233(1):228--259, 2006.

\bibitem{Biswas}
A.~Biswas, D.~Milovic, and M.~Edwards.
\newblock {\em Mathematical theory of dispersion-managed optical solitons}.
\newblock Nonlinear Physical Science. Higher Education Press, Beijing; Springer, Heidelberg, 2010.

\bibitem{series-Bourgain}
J.~Bourgain.
\newblock Periodic {K}orteweg de {V}ries equation with measures as initial data.
\newblock {\em Selecta Mathematica. New Series}, 3(2):115--159, 1997.

\bibitem{GT-dichotomy}
M.-R. Choi, Y.~Hong, and Y.-R. Lee.
\newblock Global existence versus finite time blowup dichotomy for the dispersion managed {NLS}.
\newblock {\em Nonlinear Analysis}, 251:113696, 2025.

\bibitem{solitons-positive-dispersion}
M.-R. Choi, D.~Hundertmark, and Y.-R. Lee.
\newblock Thresholds for existence of dispersion management solitons for general nonlinearities.
\newblock {\em SIAM Journal on Mathematical Analysis}, 49(2):1519--1569, 2017.

\bibitem{GT-Choi-wp}
M.-R. Choi, D.~Hundertmark, and Y.-R. Lee.
\newblock Well–posedness of dispersion managed nonlinear {S}chr\"{o}dinger equations.
\newblock {\em Journal of Mathematical Analysis and Applications}, 522(1):126938, 2023.

\bibitem{GT-with-amplification}
M.-R. Choi, Y.~Kang, and Y.-R. Lee.
\newblock On dispersion managed nonlinear {S}chr\"{o}dinger equations with lumped amplification.
\newblock {\em Journal of Mathematical Physics}, 62(7):071506, Jul 2021.

\bibitem{series-Christ}
M.~Christ.
\newblock Power series solution of a nonlinear {S}chr\"odinger equation.
\newblock In {\em Mathematical aspects of nonlinear dispersive equations}, volume 163 of {\em Annals of Mathematics Studies}, pages 131--155. Princeton University Press, Princeton, NJ, 2007.

\bibitem{CCT-ip-2}
M.~Christ, J.~Colliander, and T.~Tao.
\newblock Ill-posedness for nonlinear {S}chrodinger and wave equations.
\newblock Preprint \texttt{arXiv:0311048}.

\bibitem{CCT-ip-1}
M.~Christ, J.~Colliander, and T.~Tao.
\newblock Asymptotics, frequency modulation, and low regularity ill-posedness for canonical defocusing equations.
\newblock {\em American Journal of Mathematics}, 125(6):1235--1293, 2003.

\bibitem{DCF-history}
R.-J. Essiambre, P.~J. Winzer, and D.~F. Grosz.
\newblock Impact of {DCF} properties on system design.
\newblock {\em Journal of Optical and Fiber Communications Reports}, 3(4):221--291, Aug 2006.

\bibitem{Fibich}
G.~Fibich.
\newblock {\em The nonlinear {S}chr\"odinger equation}, volume 192 of {\em Applied Mathematical Sciences}.
\newblock Springer, Cham, 2015.

\bibitem{NLS-virial-glassey}
R.~T. Glassey.
\newblock On the blowing up of solutions to the {C}auchy problem for nonlinear schrödinger equations.
\newblock {\em Journal of Mathematical Physics}, 18(9):1794--1797, Sep 1977.

\bibitem{grafakos}
L.~Grafakos.
\newblock {\em Classical Fourier Analysis}.
\newblock Springer New York, 2014.

\bibitem{solitons-zero-dispersion}
D.~Hundertmark and Y.-R. Lee.
\newblock On non-local variational problems with lack of compactness related to non-linear optics.
\newblock {\em Journal of Nonlinear Science}, 22(1):1--38, 2012.

\bibitem{Kawakami}
J.~Kawakami.
\newblock Scattering and blow up for nonlinear {S}chr\"odinger equation with the averaged nonlinearity.
\newblock {\em Journal of Mathematical Analysis and Applications}, 543(2, Part 1):128932, 2025.

\bibitem{GT-Jason-small-large-scattering}
J.~Kawakami and J.~Murphy.
\newblock Small and large data scattering for the dispersion-managed {NLS}.
\newblock {\em Discrete and Continuous Dynamical Systems}, 47(0):256--285, 2026.

\bibitem{NLS-clay-lecture}
R.~Killip and M.~Vi\c{s}an.
\newblock Nonlinear {S}chr\"odinger equations at critical regularity.
\newblock In {\em Evolution equations}, volume~17 of {\em Clay Mathematics Proceedings}, pages 325--437. American Mathematics Society, Providence, RI, 2013.

\bibitem{series-Kishimoto}
N.~Kishimoto.
\newblock A remark on norm inflation for nonlinear {S}chr\"odinger equations.
\newblock {\em Communications on Pure and Applied Analysis}, 18(3):1375--1402, 2019.

\bibitem{origin-linear-dispersion-management}
C.~Lin, H.~Kogelnik, and L.~G. Cohen.
\newblock Optical-pulse equalization of low-dispersion transmission in single-mode fibers in the 1.3--1.7-$\mu$m spectral region.
\newblock {\em Optics Letters}, 5(11):476--478, 1980.

\bibitem{blog-post}
{Linden Photonics, Inc.}
\newblock Dispersion compensating fiber ({DCF}) for mitigating chromatic dispersion ({CD}) effects, Apr 2025.
\newblock \texttt{lindenphotonics.com}.

\bibitem{GT-justification-2}
P.~M. Lushnikov.
\newblock Dispersion-managed soliton in a strong dispersion map limit.
\newblock {\em Optics Letters}, 26(20):1535--1537, 2001.

\bibitem{GT-justification-3}
P.~M. Lushnikov.
\newblock Oscillating tails of a dispersion-managed soliton.
\newblock {\em Journal of the Optical Society of America B}, 21(11):1913--1918, 2004.

\bibitem{T-averaging}
S.~B. Medvedev and S.~K. Turitsyn.
\newblock Hamiltonian averaging and integrability in nonlinear systems with periodically varying dispersion.
\newblock {\em Journal of Experimental and Theoretical Physics Letters}, 69(7):499--504, Apr. 1999.

\bibitem{GT-Jason-large-scale}
J.~Murphy.
\newblock Large scale limit for a dispersion-managed {NLS}.
\newblock Preprint \texttt{arXiv:2504.13798}.

\bibitem{series-Oh}
T.~Oh.
\newblock A remark on norm inflation with general initial data for the cubic nonlinear {S}chr\"odinger equations in negative {S}obolev spaces.
\newblock {\em Funkcialaj Ekvacioj. Serio Internacia}, 60(2):259--277, 2017.

\bibitem{dispersion-management-survey}
S.~K. Turitsyn, B.~G. Bale, and M.~P. Fedoruk.
\newblock Dispersion-managed solitons in fibre systems and lasers.
\newblock {\em Physics Reports}, 521(4):135--203, 2012.

\bibitem{DSP-concerns-Turitsyn}
S.~Vergeles and S.~K. Turitsyn.
\newblock Optical rogue waves in telecommunication data streams.
\newblock {\em Physical Review A}, 83:061801, Jun 2011.

\bibitem{DSP-DCF-comparison}
C.~Xie.
\newblock {WDM} coherent {PDM}-{QPSK} systems with and without inline optical dispersion compensation.
\newblock {\em Optics Express}, 17(6):4815--4823, Mar 2009.

\bibitem{GT-justification-4}
V.~E. Zakharov and S.~Wabnitz.
\newblock {\em Optical Solitons: Theoretical Challenges and Industrial Perspectives: Les Houches Workshop, September 28--October 2, 1998}, volume~12.
\newblock Springer Science \& Business Media, 2013.

\bibitem{solitons-zero-dispersion-reason}
V.~Zharnitsky, E.~Grenier, C.~K. Jones, and S.~K. Turitsyn.
\newblock Stabilizing effects of dispersion management.
\newblock {\em Physica D: Nonlinear Phenomena}, 152--153:794--817, 2001.
\newblock Advances in Nonlinear Mathematics and Science: A Special Issue to Honor Vladimir Zakharov.

\end{thebibliography}
\end{document}